\documentclass[12pt]{amsart}

\usepackage[T1]{fontenc}
\usepackage{lmodern}
\usepackage{amsmath,amssymb,mathtools}
\usepackage{enumitem}
\usepackage{microtype}
\usepackage[colorlinks=true,linkcolor=blue,citecolor=blue,urlcolor=blue]{hyperref}
\hypersetup{
  pdftitle={Conformal welding of chord-arc curves},
  pdfauthor={Tailiang Liu},
  pdfsubject={Conformal weldings, analytic projections on BMOA, and chord-arc curves},
  pdfkeywords={BMOA, conformal welding, chord-arc curve, Faber operator, strongly quasisymmetric homeomorphism}
}

\allowdisplaybreaks

\newtheorem{theorem}{Theorem}[section]
\newtheorem{proposition}[theorem]{Proposition}
\newtheorem{lemma}[theorem]{Lemma}

\theoremstyle{definition}

\theoremstyle{remark}

\newcommand{\D}{\mathbb D}
\newcommand{\Ds}{\mathbb D^{*}}
\newcommand{\T}{\mathbb T}
\newcommand{\C}{\mathbb C}
\newcommand{\R}{\mathbb R}
\newcommand{\BMO}{\mathrm{BMO}}
\newcommand{\BMOA}{\mathrm{BMOA}}
\newcommand{\VMO}{\mathrm{VMO}}
\newcommand{\SQS}{\mathrm{SQS}}
\newcommand{\Ran}{\operatorname{Ran}}
\newcommand{\Ker}{\operatorname{Ker}}
\newcommand{\dist}{\operatorname{dist}}
\newcommand{\diam}{\operatorname{diam}}
\newcommand{\length}{\operatorname{length}}

\newcommand{\Faber}{\mathcal F}
\newcommand{\Bplus}{\mathcal B_{+}}
\newcommand{\Bminus}{\mathcal B_{-}}
\newcommand{\BplusGamma}{\mathcal B_{+}(\Gamma)}
\newcommand{\BminusGamma}{\mathcal B_{-}(\Gamma)}
\newcommand{\Bloch}{B}

\title[Conformal welding of chord-arc curves]
{Conformal welding of chord-arc curves}
\author{Tailiang Liu}
\address{School of Mathematics and Physics, Jiangsu University of Technology, Changzhou, Jiangsu, P. R. China}
\email{Ltlmath@jsut.edu.cn}
\thanks{This work was supported by the National Natural Science Foundation of China under Grants Nos. 12401095}
\date{}

\begin{document}

\begin{abstract}
We study the relation between the geometric properties of a chord-arc curve and its conformal welding. Let $h$ be the conformal welding of a closed  quasicircle $\Gamma$. 
By Jones’s theorem, the pull-back operator $C_h u=u\circ h$ is bounded on BMO if and only if $h$ corresponds to
the welding of a Bishop-Jones quasicircle, equivalently, $h$ is strongly quasisymmetric.  Let $A_h$ denote the analytic projection of $C_h$.
We prove that  $A_h$ is a bounded isomorphism on BMOA if and only if $\Gamma$ is a chord-arc curve. More strongly, the same characterization holds if invertibility is replaced by Fredholmness. This gives an intrinsic conformal-welding characterization of chord-arc curves and a complete geometric answer to the invertibility problem posed by Semmes in the 1980s. Furthermore, we establish an exact correspondence between the inverse of $A_h$ and the classical Faber integral operator, showing that for a rectifiable quasicircle, the Faber operator is a bounded isomorphism on BMOA if
and only if the curve satisfies the chord-arc condition.
\end{abstract}

\maketitle

\section{Introduction and main results}\label{sec:introduction}
 Let $\mathbb{D}$ denote the open unit disk, $\mathbb{D}^* = \overline{\mathbb{C}} \setminus \overline{\mathbb{D}}$ its exterior in the Riemann sphere, and $\mathbb{T} = \partial\mathbb{D}$ the unit circle. Let $\Gamma$ be a bounded Jordan curve in the complex plane $\mathbb{C}$, and denote by $\Omega_{+}$ and $\Omega_{-}$ its bounded and unbounded complementary components.
Let $g: \mathbb{D} \to \Omega_+$ and $f: \mathbb{D}^* \to \Omega_-$ be conformal mappings, with $f(\infty)=\infty$. 
 By Carath\'eodory's theorem, both maps extend homeomorphically to the boundary $\mathbb{T}$. Their boundary correspondence\begin{equation}\label{eq:welding-intro-new}h=f^{-1}\circ g\big|_{\mathbb{T}}\end{equation}is called the \emph{conformal welding} of $\Gamma$. A basic problem of conformal welding asks how geometric properties of $\Gamma$ are reflected in analytic properties of $h$. The classical quasi-conformal mapping theory shows that  $\Gamma$ is a quasicircle if and only if its welding is quasisymmetric (see \cite{Ahlfors,BA}). Here, a homeomorphism $h$ on $\mathbb{T}$ is quasisymmetric and belongs to the class $\text{QS}(\mathbb{T})$  if there exists a positive constant $C$ such that for any interval $I\subset \mathbb{T}$ of length less than $\pi$, we have$$\vert{}h(2I)\vert{}\leq C\vert{}h(I)\vert{},$$ where $2I$ is the interval with the same center as $I$ but $\vert{}2I\vert{}=2\vert{}I\vert{}$. The next natural classes to consider are chord-arc curves,  which was asked by Jerison and Kenig (see \cite[p.246]{Jk}). A rectifiable Jordan curve is called \emph{chord-arc curve} if there is $M\ge1$ such that, for every $z_1,z_2\in\Gamma$, the shorter subarc $\Gamma(z_1,z_2)$ satisfies\begin{equation}\label{eq:chord-arc-definition}\text{length}\big(\Gamma(z_1,z_2)\big)\le M|z_1-z_2|.\end{equation}
A   homeomorphism $h$  is called strongly quasisymmetric  in the sense of Semmes \cite{Semmes} and belongs to the class SQS$(\mathbb{T})$, if there exist $0<C_{1}, C_{2}<\infty$ such that for all intervals $I$ in $\mathbb{T}$  and all measurable subsets $A$ of $I$ we have
\[\frac{|h(A)|}{|h(I)|}<C_{1}\Big(\frac{|A|}{|I|}\Big)^{C_{2}}.\] Equivalently, $h\in\text{SQS}(\mathbb{T})$   if and only if $h$ is locally absolutely continuous so that $h'$ belongs to the class $A_{\infty}$ weights  introduced by Muckenhoupt (see \cite{CF}), in particular, $\log h'\in$ BMO$(\mathbb{T})$,  the space  of functions on
$\mathbb{T}$ of bounded mean oscillation. 
Lavrentiev's theorem implies that, for a chord-arc curve, the boundary derivatives of the conformal maps are $A_\infty$ weights. It follows from the welding $h$ is strongly quasisymmetric. The converse fails dramatically. Semmes \cite{SemmesCounterexample}constructed a bi-Lipschitz welding whose welded Jordan curve is not locally rectifiable (see also Bishop's \cite{Bishop}). Hence the condition
\[
 h\in\SQS(\T),\qquad\text{equivalently}\qquad
 h\text{ is absolutely continuous and }h'\in A_\infty,
\]
contains the chord-arc class but does not characterize it. In fact, the strongly quasisymmetric weldings correspond precisely to Bishop--Jones quasicircles (see \cite{AstalaZinsmeister,BJ, FeffermanKenigPipher,Semmes})
In the small-constant regime, David's theorem \cite{David}, together with the work of Lavrentiev and Pommerenke \cite{PommerenkeSchlichte}, gives a quantitative correspondence between chord-arc curves with small constant and weldings for which $\|\log h'\|_{\BMO}$ is small. Finer subclasses are now much better understood. For example, asymptotically smooth and Weil--Petersson curves admit descriptions through VMO, fractional Sobolev spaces; see \cite{LiuShenVMO, ShenWP, ShenTang,ShenWei,ShenWu}.
The global chord-arc locus remains subtler, concerning the connectedness of the chord-arc space in the BMO-Teichm\"uller space and the characterization of its corresponding complex  dilatation; see \cite{AstalaGonzalez,AstalaZinsmeister,Bishop2,CoifmanMeyer}.

The operator considered here arises from the BMO pull-back
\[
 C_hu=u\circ h.
\]
For $h\in\SQS(\T)$, the pull-back $C_h$ is a bounded isomorphism on BMO \cite{Jones}.

Decompose boundary BMO into its analytic and anti-analytic trace spaces,
\[
 \BMO(\T)=\Bplus\dotplus\Bminus,
\] where $\Bplus$ consists of the boundary values on $\mathbb{T}$ of analytic functions in $\text{BMOA}(\mathbb{D})$, and $\Bminus$ represents the boundary values of their anti-analytic counterparts $\overline{\text{BMOA}}(\mathbb{D})$.
Let $P_{+}$ and $P_{-}$ be the corresponding Riesz projections. The two  projected composition operators are
\begin{equation*}
\label{eq:Ah-def} A_h=
P_{+}C_h\big|_{\Bplus}
\quad\text{and}\quad
B_h=P_{-}C_h\big|_{\Bplus}.
\end{equation*}

Semmes asked when the operator  $A_h$ is invertible \cite[p.424]{SemmesQuestion}. Bishop's counterexample \cite{Bishop} shows that strong quasisymmetry alone does not prevent a nontrivial kernel, while David's theorem settles the problem when $\|\log h'\|_{\BMO}$ is sufficiently small.   Fan-Hu-Shen conjectured in \cite{FanHuShen} that the condition $\log h' \in \VMO$, which geometrically means that the curve $\Gamma$ is asymptotically smooth, implies the invertibility of this operator.  Furthermore, they also conjectured that the operator $B_h$ is compact if and only if the curve is asymptotically smooth, which have be discussed  in  \cite{LiuAsymptotically}. 

Our first purpose is to give a complete result for chord-arc curves
\begin{theorem}[Main characterization]\label{thm:main-characterization}
Let $\Gamma$, $f$, $g$ and $h$ be as in \eqref{eq:welding-intro-new}, and assume $h\in\SQS(\T)$. Then the following are equivalent:
\begin{enumerate}[label=\textup{(\roman*)}]
\item $\Gamma$ is a chord-arc curve;
\item $A_h:\Bplus\to\Bplus$ is a bounded isomorphism;
\item $A_{h^{-1}}:\Bplus\to\Bplus$ is a bounded isomorphism.
\end{enumerate}
More sharply, if there is $c>0$ such that
\begin{equation}\label{eq:lower-Ahinv}
 \|A_{h^{-1}}\varphi\|_{\BMOA}
 \ge c\|\varphi\|_{\BMOA},
 \qquad \varphi\in\Bplus,
\end{equation}
then $\Gamma$ is chord-arc.
\end{theorem}

This result provides a global conformal-welding characterization of chord-arc curves, resolving a question posed by Jerison and Kenig. It also gives a complete answer to Semmes’s invertibility problem. David’s small-BMO theorem emerges as a special case: when \(\log h'\) has sufficiently small BMO norm, the welding curve is chord-arc, which immediately implies the invertibility of \(A_h\). In particular, it settles the conjecture of Fan, Hu and Shen \cite{FanHuShen}: if \(\log h'\in\VMO\), then the associated curve is asymptotically smooth and hence chord-arc, so \(A_h\) is invertible.

 The proof extracts the missing geometry directly. For every finite ordered partition of a boundary subarc we construct a logarithmic jump function across the subarc whose derivative contains the  variation $|z_{j+1}-z_j|$. Estimate \eqref{eq:lower-Ahinv} gives a uniform BMOA bound independent of the partition, which allows us to control the variation at one fixed interior point, and hence establish rectifiability. Localizing the same functions to an interior comparison ball at the scale of the target arc gives the chord-arc estimate. 

It is worth noting that the proof method of the main theorem remains applicable, which allows us to obtain the stronger result below.
\begin{theorem}[Fredholm rigidity]\label{cor:fredholm-intro}
Under the hypotheses of Theorem~\ref{thm:main-characterization}, the
following conditions are also equivalent to the chord-arc property of $\Gamma$:
\begin{enumerate}[label=\textup{(\roman*)}]
\item $A_h$ is a Fredholm operator on $\Bplus$;
\item $A_{h^{-1}}$ is a Fredholm operator on $\Bplus$;
\item at least one of $A_h$ and $A_{h^{-1}}$ has closed range and a
finite-dimensional kernel.
\end{enumerate}
\end{theorem}
Here a bounded operator is Fredholm if its range is closed and both its
kernel and its cokernel are finite-dimensional. The weaker condition in
\textup{(iii)} is usually called upper semi-Fredholmness.   Theorem \ref{thm:main-characterization} gives a complete characterization of chord‑arc curves via invertibility, and Theorem \ref{cor:fredholm-intro} provides a finer identification of the failure mechanism for the chord‑arc property.
It also suggests a potentially useful route toward the connectedness problem for the chord‑arc curves in the BMO‑Teichm\"uller space. Namely, one may study paths of weldings and ask whether Fredholmness can be controlled along such paths in the BMO topology.

The second purpose of the paper is to identify the inverse analytic projection with the classical Faber operator. Once $\Gamma$ is rectifiable, define the interior Faber operator by
\begin{equation}\label{eq:Faber-intro-new}
 \Faber_f^{+}\varphi(w)
 =\frac{1}{2\pi i}\int_{\T}
 \frac{f'(\zeta)\varphi(\zeta)}{f(\zeta)-w}\,d\zeta,
 \qquad w\in\Omega_{+},
\end{equation}
where the integral is understood through the $H^1$--BMOA pairing.  For analytic polynomials, the same  integral also has a classical exterior branch on $\Omega_{-}$.

Recall that $C_g F = F\circ g$, and $\BMOA(\Omega_{+})$ is the space of holomorphic functions on $\Omega_+$ obtained by conformal pull‑back of $g$.
By adapting the main‑theorem arguments to a continuous setting, we show that boundedness of the Faber operator forces the chord‑arc estimate for every rectifiable quasicircle.  The exact identity is established for polynomials via the Cauchy jump formula, and then extended to \(\mathrm{BMOA}\) by weak‑star approximation.

\begin{theorem}[Faber characterization]\label{thm:Faber-characterization-intro}
Let $\Gamma$ be a bounded rectifiable quasicircle, with conformal maps
$f$, $g$ and welding $h$ as in \eqref{eq:welding-intro-new}. Then the
following are equivalent:
\begin{enumerate}[label=\textup{(\roman*)}]
\item $\Gamma$ is chord-arc;
\item $\Faber_f^{+}:\BMOA(\D)\to\BMOA(\Omega_{+})$ is bounded;
\item $\Faber_f^{+}:\BMOA(\D)\to\BMOA(\Omega_{+})$ is a bounded isomorphism.
\end{enumerate}
Whenever these conditions hold, $h$ is strongly quasisymmetric,
$A_h$ and $A_{h^{-1}}$ are bounded isomorphisms, and
\begin{equation}\label{eq:main-exact-identity}
 \boxed{\quad C_g\Faber_f^{+}=(A_{h^{-1}})^{-1}.\quad}
\end{equation}
\end{theorem}

The paper is organized as follows. Section~\ref{sec:spaces} fixes the Hardy, BMO, BMOA and Bloch-space notation. Section~\ref{sec:projection} gives the transversality interpretation of $A_h$ and proves the direct implication from the BMO jump decomposition on a chord-arc curve. Sections~\ref{sec:logarithmic-jumps} and \ref{sec:alternative-whitney} establish the converse direction of the main theorem; Section~\ref{sec:alternative-whitney} further derives the Fredholm‑type characterization. Section~\ref{sec:classical-faber} treats the Faber operator, proving the corresponding converse statement as well as exact operator identities and equivalence relations. 

Throughout the paper, $C,C_1,C_2,\ldots$ denote positive constants whose values may change from line to line. Their dependence is indicated when it matters. We write $X\lesssim Y$ if $X\le CY$, $X\gtrsim Y$ if $X\ge C^{-1}Y$, and $X\asymp Y$ if both estimates hold. The Euclidean disk with center $z$ and radius $r$ is denoted by $B(z,r)$.

\section{Preliminaries}\label{sec:spaces}

\subsection{Hardy spaces and boundary decompositions}
For $0<p<\infty$, the Hardy space $H^p(\D)$ consists of holomorphic functions $F$ on $\D$ for which
\[
 \|F\|_{H^p(\D)}^p
 =\sup_{0<r<1}\frac1{2\pi}\int_0^{2\pi}|F(re^{it})|^p\,dt<\infty,
\]
with the usual modification for $p=\infty$. The exterior space $H^p(\Ds)$ is defined analogously by radial limits from $|z|>1$. We use $H_0^p(\Ds)$ for the subspace of functions vanishing at infinity. Boundary values are always understood almost everywhere.

For $u\in L^1(\T)$ and an arc $I\subset\T$, set
\[
 u_I=\frac1{|I|}\int_Iu\,|d\zeta|,
 \qquad
 \|u\|_{\BMO(\T)}
 =\sup_{I\subset\T}\frac1{|I|}\int_I|u-u_I|\,|d\zeta|.
\]
Throughout the paper, $\BMO$ and $\BMOA$ are understood modulo constants; this convention is made once here. Whenever an actual function norm is needed, one may adjoin a point value, for example $|F(0)|+\|F\|_{\BMOA}$. Composition and interior Faber operators preserve constants, while the Riesz projections send constants to constants or zero; hence these operators induce well-defined maps on the quotient spaces. When actual boundary functions are used, $P_+$ retains the zero Fourier coefficient and $P_-$ contains only strictly negative frequencies. Pointwise statements use fixed representatives, and a quotient identity is understood modulo constants.

Recall that  $\Bplus$ is the boundary trace space of $\BMOA(\D)$, $\Bminus$ is the boundary trace space of $\BMOA_0(\Ds)$. In Fourier language,
\[
 \Bplus=\{u\in\BMO:\widehat u(n)=0\text{ for }n<0\}/\C,
 \qquad
 \Bminus=\{u\in\BMO:\widehat u(n)=0\text{ for }n\ge0\}.
\]
The Hilbert transform is bounded on BMO, and hence
\begin{equation}\label{eq:BMO-direct-sum}
 \BMO(\T)=\Bplus\dotplus\Bminus.
\end{equation}
We denote the corresponding bounded Riesz projections by
\[
 P_{+}:\BMO(\T)\to\Bplus,
 \qquad
 P_{-}:\BMO(\T)\to\Bminus.
\]
Thus
\begin{equation}\label{eq:Pplus-kernel-range}
 \Ran P_{+}=\Bplus,
 \qquad
 \Ker P_{+}=\Bminus,
 \qquad
 P_{+}+P_{-}=I.
\end{equation}

\subsection{Equivalent descriptions of BMOA and the Bloch embedding}
For $a\in\D$, let
\[
 \phi_a(z)=\frac{a-z}{1-\overline a z}.
\]
For a holomorphic function $F$ on $\D$, the following conditions are equivalent:
\begin{enumerate}[label=\textup{(\roman*)}]
\item $F\in H^2(\D)$ and its boundary values belong to $\BMO(\T)$;
\item
\begin{equation}\label{eq:mobius-BMOA}
 \sup_{a\in\D}\|F\circ\phi_a-F(a)\|_{H^2(\D)}<\infty;
\end{equation}
\item the measure
\begin{equation}\label{eq:Carleson-measure-BMOA}
 d\mu_F(z)=|F'(z)|^2(1-|z|^2)\,dA(z)
\end{equation}
is a Carleson measure, equivalently
\begin{equation}\label{eq:Carleson-BMOA-audited}
 \sup_{I\subset\T}\frac1{|I|}
 \int_{S(I)}|F'(z)|^2(1-|z|^2)\,dA(z)<\infty.
\end{equation}
\end{enumerate}
Here $S(I)$ is any standard Carleson box above $I$. The quantity in \eqref{eq:mobius-BMOA}, the square root of the quantity in \eqref{eq:Carleson-BMOA-audited}, and the boundary BMO seminorm give equivalent seminorms on $\BMOA(\D)$; see \cite{Garnett}.

The Bloch space $\Bloch(\D)$ consists of holomorphic functions satisfying
\[
 \|F\|_{\Bloch}
 =\sup_{z\in\D}(1-|z|^2)|F'(z)|<\infty.
\]
The standard embedding $\BMOA(\D)\hookrightarrow\Bloch(\D)$ gives
\begin{equation}\label{eq:Bloch-BMOA-audited}
 (1-|z|^2)|F'(z)|\le C\|F\|_{\BMOA},
 \qquad z\in\D.
\end{equation}

For a holomorphic function $G$ on $\Ds$ with $G(\infty)=0$, put
\[
 \widetilde G(z)=\overline{ G(1/\bar{z})},\qquad z\in\D.
\]
We say that $G\in\BMOA_0(\Ds)$ if $\widetilde G\in\BMOA(\D)$. This definition is equivalent to requiring the exterior boundary trace of $G$ to belong to $\Bminus$. The analogous definitions will be used on simply connected Jordan domains by conformal pull-back.

We shall repeatedly use a local form of the Carleson-measure estimate.

\begin{lemma}[Uniform Dirichlet energy on hyperbolic disks]\label{lem:hyperbolic-energy-audited}
For every fixed hyperbolic radius $R>0$ there is $C_R<\infty$ such that
\begin{equation}\label{eq:hyperbolic-energy-audited}
 \int_{B_{\mathrm{hyp}}(z,R)}|F'(\zeta)|^2\,dA(\zeta)
 \le C_R\|F\|_{\BMOA}^2
\end{equation}
for every $z\in\D$ and every $F\in\BMOA(\D)$.
\end{lemma}

\begin{proof}
Let $\phi_z$ be a disk automorphism with $\phi_z(0)=z$. Conformal invariance of the Dirichlet integral gives
\[
 \int_{B_{\mathrm{hyp}}(z,R)}|F'|^2\,dA
 =\int_{B_{\mathrm{hyp}}(0,R)}|(F\circ\phi_z)'|^2\,dA.
\]
Subtracting $F(z)$ does not change the derivative. If $r_R<1$ is the Euclidean radius of $B_{\mathrm{hyp}}(0,R)$ and
\[
 F\circ\phi_z-F(z)=\sum_{n\ge1}a_n\zeta^n,
\]
then
\[
 \int_{|\zeta|<r_R}|(F\circ\phi_z)'(\zeta)|^2\,dA(\zeta)
 =\pi\sum_{n\ge1}n|a_n|^2r_R^{2n}
 \le C_R\sum_{n\ge1}|a_n|^2.
\]
The last sum is bounded by the square of the M\"obius-invariant BMOA seminorm in \eqref{eq:mobius-BMOA}.
\end{proof}

\section{Analytic projections and chord-arc curves}\label{sec:projection}
This section proves the implication
\[
\Gamma\text{ chord-arc}\quad\Longrightarrow\quad A_h\text{ is a bounded isomorphism},
\] which relies crucially on the jump decomposition of $\text{BMO}$ functions on a chord-arc curve established in the our   previous work \cite{LiuShenCauchy,LiuShenFaber}. 
\subsection{A transversality theorem}

The next  result gives the transversality interpretation of invertibility.
\begin{theorem}\label{thm:transversality}
For $h\in\SQS(\T)$, the following conditions are equivalent:
\begin{enumerate}[label=\textup{(\roman*)}]
\item $A_h:\Bplus\to \Bplus$ is a bounded isomorphism;
\item
\begin{equation}\label{eq:transversality}
 \BMO(\T)=C_h(\Bplus)\dotplus \Bminus;
\end{equation}
\item the restriction
\[
 P_{+}|_{C_h(\Bplus)}:C_h(\Bplus)\longrightarrow \Bplus
\]
is a bounded isomorphism.
\end{enumerate}
If these conditions hold, the bounded projection onto $C_h(\Bplus)$ along $\Bminus$ is
\begin{equation}\label{eq:Qh}
 Q_h=C_hA_h^{-1}P_{+}.
\end{equation}
\end{theorem}

\begin{proof}
Assume first that $A_h$ is invertible.  Given $u\in\BMO$, define
\[
 \varphi=A_h^{-1}P_{+}u,
 \qquad
 v_{-}=u-C_h\varphi.
\]
Then
\[
 P_{+}v_{-}=P_{+}u-P_{+}C_hA_h^{-1}P_{+}u=0,
\]
so $v_{-}\in \Bminus$ and
\[
 u=C_h\varphi+v_{-}.
\]
If $C_h\varphi+v_{-}=0$, applying $P_{+}$ gives $A_h\varphi=0$, hence $\varphi=v_{-}=0$.  Thus the sum is algebraically direct.  Formula \eqref{eq:Qh} shows that its coordinate projection is bounded, so it is a topological direct sum.

The equivalence of (ii) and (iii) follows because $P_{+}$ has kernel $\Bminus$.  More explicitly, under \eqref{eq:transversality}, $P_{+}|_{C_h(\Bplus)}$ is injective.  Given $a\in \Bplus$, write $a=C_h\varphi+v_{-}$; then $a=P_{+}C_h\varphi$, so the restriction is surjective.  Its inverse is $Q_h|_{\Bplus}$.

Finally,
\[
 A_h=\bigl(P_{+}|_{C_h(\Bplus)}\bigr)\circ\bigl(C_h|_{\Bplus}\bigr),
\]
which proves the remaining implication.
\end{proof}

\begin{proposition}\label{prop:inverse-symmetry}
For $h\in\SQS(\T)$,
\[
 A_h\text{ is invertible}\quad\Longleftrightarrow\quad
 A_{h^{-1}}\text{ is invertible}.
\]
More generally, their kernels have the same dimension, their ranges have
the same codimension, and one range is closed if and only if the other is
closed. In particular, Fredholmness and upper semi-Fredholmness are
unchanged when $h$ is replaced by $h^{-1}$.
\end{proposition}

\begin{proof}
Write $Z=\BMO(\T)/\C$, $N=\Bminus$, and $E_h=C_h(\Bplus)$.
The isomorphism $C_h:\Bplus\to E_h$ identifies
\[
 \Ker A_h\simeq E_h\cap N,
 \qquad
 \Ran A_h=P_+(E_h).
\]
Since $\Ker P_+=N$, one also has the algebraic direct-sum identity
\begin{equation}\label{eq:defect-sum}
 E_h+N=\Ran A_h\dotplus N.
\end{equation}
Thus $E_h+N$ is closed in $Z$ if and only if $\Ran A_h$ is closed in
$\Bplus$: one implication follows by intersection with $\Bplus$, and
the other from the fixed topological splitting $Z=\Bplus\dotplus N$.
Moreover, $P_+$ induces an isomorphism of the quotient vector spaces
\[
 Z/(E_h+N)\simeq\Bplus/\Ran A_h.
\]

Let $\mathcal J u=\overline u$ on $Z$. Complex conjugation interchanges
$\Bplus$ and $N$ and commutes with every real circle composition operator.
The bounded conjugate-linear isomorphism $\mathcal J C_{h^{-1}}$ sends
$E_h$ to $N$ and $N$ to $E_{h^{-1}}$. It therefore sends the intersection
$E_h\cap N$ onto $N\cap E_{h^{-1}}$ and the sum $E_h+N$ onto
$N+E_{h^{-1}}$. This preserves their dimensions, codimensions, and
closedness, proving the assertions about kernels and ranges.

For invertibility, the same argument sends the topological direct sum
$Z=E_h\dotplus N$ to $Z=E_{h^{-1}}\dotplus N$.
Theorem~\ref{thm:transversality} gives the stated equivalence.
\end{proof}

\subsection{BMO space on chord-arc curve}
Assume throughout this section that $\Gamma$ is chord-arc.  Let $ds$ denote arclength measure and define
\[
 \|u\|_{\BMO(\Gamma)}
 =\sup_{E\subset\Gamma}\frac1{|E|}\int_E|u-u_E|\,ds,
 \qquad
 u_E=\frac1{|E|}\int_Eu\,ds,
\]
where the supremum is taken over connected subarcs $E$ and $|E|$ denotes arclength.  As above, $\BMO(\Gamma)$ is understood modulo constants.

Let $L=\length\Gamma$, and let
\[
 \sigma:\T\longrightarrow\Gamma
\]
be the positively oriented arclength parametrization.  Thus, writing $\zeta=e^{it}$,
\[
 \left|\frac{d}{dt}\sigma(e^{it})\right|=\frac{L}{2\pi}
 \quad\text{for a.e. }t.
\]
It is easy to show that
\begin{equation}\label{eq:Usigma-isometry-new}
 U_\sigma:\BMO(\Gamma)\longrightarrow\BMO(\T),
 \qquad U_\sigma u=u\circ\sigma,
\end{equation}
is an isometric isomorphism.

There are unique orientation-preserving circle homeomorphisms $\alpha,\beta$ such that
\begin{equation}\label{eq:gf-sigma-alpha-beta}
 g|_{\T}=\sigma\circ\alpha,
 \qquad
 f|_{\T}=\sigma\circ\beta.
\end{equation}
Lavrentiev's theorem for chord-arc curves \cite{Pommerenke,Semmes} gives absolute continuity of the boundary extensions and
\begin{equation}\label{eq:Lavrentiev-both}
 |g'|\in A_\infty(\T),
 \qquad
 |f'|\in A_\infty(\T).
\end{equation}
Let $A,B:\R\to\R$ be increasing lifts of $\alpha,\beta$.  Differentiating \eqref{eq:gf-sigma-alpha-beta} at almost every point gives
\begin{equation}\label{eq:alpha-beta-derivatives-new}
 A'(t)=\frac{2\pi}{L}|g'(e^{it})|,
 \qquad
 B'(t)=\frac{2\pi}{L}|f'(e^{it})|.
\end{equation}
Consequently $A'$ and $B'$ are $A_\infty$ weights, so
\begin{equation}\label{eq:alpha-beta-SQS-new}
 \alpha,\beta\in\SQS(\T).
\end{equation}
The welding relation $g=f\circ h$ and \eqref{eq:gf-sigma-alpha-beta} imply
\[
 \sigma\circ\alpha=\sigma\circ\beta\circ h,
\]
and therefore
\begin{equation}\label{eq:h-beta-alpha}
 h=\beta^{-1}\circ\alpha,
 \qquad h^{-1}=\alpha^{-1}\circ\beta.
\end{equation}
In particular a chord-arc welding is strongly quasisymmetric.

Define the exterior conformal pull-back
\begin{equation}\label{eq:Uf-def-new}
 U_f:\BMO(\Gamma)\longrightarrow\BMO(\T),
 \qquad U_fu=u\circ f|_{\T}.
\end{equation}
By \eqref{eq:gf-sigma-alpha-beta},
\[
 U_fu=(u\circ\sigma)\circ\beta=C_\beta U_\sigma u.
\]
Since $C_\beta$ and $C_{\beta^{-1}}$ are bounded BMO automorphisms, $U_f$ is a bounded isomorphism. More precisely,
\begin{equation}\label{eq:BMO-curve-pullback-equivalence}
 \|u\circ f\|_{\BMO(\T)}
 =\|C_\beta U_\sigma u\|_{\BMO(\T)}
 \asymp\|U_\sigma u\|_{\BMO(\T)}
 =\|u\|_{\BMO(\Gamma)}.
\end{equation}
Thus intrinsic BMO on $\Gamma$ and BMO pulled back by the exterior conformal parameter are equivalent as Banach spaces, not merely as sets. The factorization and its inverse are
\begin{equation}\label{eq:Uf-factorization-new}
 U_f=C_\beta U_\sigma,
 \qquad
 U_f^{-1}=U_\sigma^{-1}C_{\beta^{-1}}.
\end{equation}

Define
\[
 \BMOA(\Omega_{+})
 =\{F\in\mathcal O(\Omega_{+}):F\circ g\in\BMOA(\D)\},
\]
and
\[
 \BMOA_0(\Omega_{-})
 =\{G\in\mathcal O(\Omega_{-}):G\circ f\in\BMOA_0(\Ds),\ G(\infty)=0\}.
\] Here, $\mathcal{O}(\Omega)$ denotes the space of all analytic functions on $\Omega$. Then these spaces have non-tangential limit values almost everywhere in $\Gamma$ with respect to the arc-length
measure. 
Denote their boundary trace spaces by $\BplusGamma$ and $\BminusGamma$, respectively. 
Then $\BplusGamma$ and $\BminusGamma$ are subspaces of $\BMO(\Gamma)$; more details can be found in our paper \cite{LiuShenFaber}. 

\subsection{The BMO jump theorem}
 The following result is a consequence of the boundedness of the Cauchy singular integral on $\BMO(\Gamma)$ (see  \cite{LiuShenCauchy}.)  
\begin{theorem}
If $\Gamma$ is chord-arc, then
\begin{equation}\label{eq:curve-BMO-direct-sum-new}
 \BMO(\Gamma)=\BplusGamma\dotplus\BminusGamma
\end{equation}
as a topological direct sum.  Equivalently, there is a bounded projection
\[
 Q_\Gamma^{+}:\BMO(\Gamma)\longrightarrow\BplusGamma
\]
whose kernel is $\BminusGamma$.
\end{theorem}

\subsection{Pull-back of the two trace spaces}
\begin{proposition}\label{prop:pullback-curve-spaces-new}
Under the exterior pull-back $U_f$,
\begin{equation}\label{eq:Uf-Aplus-new}
 U_f(\BplusGamma)=C_{h^{-1}}(\Bplus),
\end{equation}
and
\begin{equation}\label{eq:Uf-Aminus-new}
 U_f(\BminusGamma)=\Bminus.
\end{equation}
\end{proposition}

\begin{proof}
Let $F\in\BMOA(\Omega_{+})$.  Write $F=\varphi\circ g^{-1}$ with $\varphi\in\BMOA(\D)$.  Since $f=g\circ h^{-1}$ on $\T$, for almost every $\zeta\in\T$,
\begin{align*}
 (U_fF^*)(\zeta)
 &=F^*(f(\zeta))\\
 &=\varphi(g^{-1}(f(\zeta)))\\
 &=\varphi(h^{-1}(\zeta))\\
 &=(C_{h^{-1}}\varphi)(\zeta).
\end{align*}
This proves $U_f(\BplusGamma)\subset C_{h^{-1}}(\Bplus)$.  Conversely, for every $\varphi\in\BMOA(\D)$ the function $F=\varphi\circ g^{-1}$ belongs to $\BMOA(\Omega_{+})$ and the same calculation gives $U_fF^*=C_{h^{-1}}\varphi$.  Hence \eqref{eq:Uf-Aplus-new} holds.

If $G\in\BMOA_0(\Omega_{-})$, write $G=\Psi\circ f^{-1}$ with $\Psi\in\BMOA_0(\Ds)$.  Then
\[
 (U_fG^*)(\zeta)=G^*(f(\zeta))=\Psi(\zeta)
\]
for almost every $\zeta$, so $U_f(\BminusGamma)\subset \Bminus$.  Conversely every element of $\Bminus$ is the trace of a unique $\Psi\in\BMOA_0(\Ds)$, and $G=\Psi\circ f^{-1}$ gives the reverse inclusion.  This proves \eqref{eq:Uf-Aminus-new}.
\end{proof}

Applying $U_f$ to \eqref{eq:curve-BMO-direct-sum-new} yields
\begin{equation}\label{eq:circle-direct-sum-hinv-new}
 \boxed{\quad
 \BMO(\T)=C_{h^{-1}}(\Bplus)\dotplus \Bminus.
 \quad}
\end{equation}
By Theorem~\ref{thm:transversality}, this is equivalent to invertibility of $A_{h^{-1}}$.  Proposition~\ref{prop:inverse-symmetry} then gives invertibility of $A_h$.

For later reference, define
\begin{equation}\label{eq:Qhinv-geometric}
 Q_{h^{-1}}=U_fQ_\Gamma^{+}U_f^{-1}.
\end{equation}
It is the bounded projection onto $C_{h^{-1}}(\Bplus)$ along $\Bminus$.  The proof of Theorem~\ref{thm:transversality} gives the explicit inverse
\begin{equation}\label{eq:Ah-inverse-geometric}
 A_{h^{-1}}^{-1}=C_hQ_{h^{-1}}\big|_{\Bplus}.
\end{equation}
In particular,
\begin{equation}\label{eq:quantitative-direct-chordarc}
 \|\varphi\|_{\BMOA}
 \le \|C_h\|\,\|Q_{h^{-1}}\|\,
 \|A_{h^{-1}}\varphi\|_{\BMOA}.
\end{equation}
We have proved the following theorem.

\begin{theorem}\label{thm:chordarc-implies-projection-direct}
If $\Gamma$ is chord-arc, then $A_h$ and $A_{h^{-1}}$ are bounded isomorphisms on $\Bplus$, equivalently on BMOA modulo constants.
\end{theorem}

\section{Logarithmic jumps and rectifiability}\label{sec:logarithmic-jumps}

\subsection{A two-point logarithmic jump}
We now prove the converse without assuming that $\Gamma$ is rectifiable. Let $a,b\in\Gamma$, $a\ne b$, and let $E(a,b)$ be the closed boundary subarc traversed from $a$ to $b$ in the orientation induced by $f|_{\T}$.

\begin{lemma}[Two-point logarithmic jump]\label{lem:two-point-jump-audited}
The domain $\widehat\C\setminus E(a,b)$ admits a single-valued analytic branch
\begin{equation}\label{eq:two-point-log-definition}
 l_{a,b}(w)=\log\frac{w-b}{w-a},
 \qquad l_{a,b}(\infty)=0.
\end{equation}
Write $l_{a,b}^{+}$ and $l_{a,b}^{-}$ for its restrictions to $\Omega_{+}$ and $\Omega_{-}$. With the orientation convention above, their non-tangential boundary values satisfy
\begin{equation}\label{eq:two-point-jump-audited}
 (l_{a,b}^{+})^*-(l_{a,b}^{-})^*
 =2\pi i\,\chi_{E(a,b)}
\end{equation}
away from $a$ and $b$, while
\begin{equation}\label{eq:two-point-derivative-audited}
 (l_{a,b}^{+})'(w)
 =\frac{b-a}{(w-b)(w-a)}.
\end{equation}
Moreover,
\[
 l_{a,b}^{+}\circ g\in\BMOA(\D),
 \qquad
 l_{a,b}^{-}\circ f\in\BMOA_0(\Ds).
\]
\end{lemma}

\begin{proof}
The complement in the Riemann sphere of a closed Jordan arc is simply connected. The function $(w-b)/(w-a)$ has neither a zero nor a pole there and tends to $1$ at infinity, so it has a unique analytic logarithm normalized by $l_{a,b}(\infty)=0$. Its restrictions to the two complementary domains are therefore the two branches required above.

Across the cut $E(a,b)$, continuation from the exterior side to the interior side makes one positive turn around the origin in the image of $(w-b)/(w-a)$. Hence the argument increases by $2\pi$, which gives \eqref{eq:two-point-jump-audited}. On the other component of $\Gamma\setminus\{a,b\}$ no cut is crossed, so the two boundary values agree. Differentiating \eqref{eq:two-point-log-definition} gives \eqref{eq:two-point-derivative-audited}.

Only the BMOA membership requires an analytic input. The functions
\[
 \frac{g-b}{g-a}\quad\text{and}\quad
 \frac{f-b}{f-a}
\]
are zero-free univalent functions in the corresponding disk charts. Girela's theorem on logarithms of zero-free univalent functions \cite{Girela} implies that their analytic logarithms belong to BMOA. For the exterior function one uses the analytic disk chart obtained by composing with $z\mapsto 1/\bar z$ and then taking complex conjugates, together with the normalization at infinity.   This proves the final assertion.
\end{proof}

\subsection{Finite jump data}\label{sec:discrete-jumps}
Fix a positively oriented exterior parameter arc $I\subset\T$ and a finite ordered partition
\[
 \zeta_0,\zeta_1,\ldots,\zeta_N.
\]
Put
\[
 z_j=f(\zeta_j),
 \qquad
 \tau_j=\frac{z_{j+1}-z_j}{|z_{j+1}-z_j|},
 \qquad 0\le j<N.
\]
For this fixed partition, define
\begin{equation}\label{eq:uP-audited}
 u(\zeta)=
 \begin{cases}
 \overline{\tau_j},&\zeta\in[\zeta_j,\zeta_{j+1}],\\
 0,&\zeta\notin I.
 \end{cases}
\end{equation}
Then $|u|\le1$, and therefore
\begin{equation}\label{eq:uP-bound-audited}
 \|u\|_{\BMO}\le2,
 \qquad
 \|P_{+}u\|_{\BMOA}\le C.
\end{equation}

Let $l_j^{\pm}=l_{z_j,z_{j+1}}^{\pm}$ be the logarithms from Lemma~\ref{lem:two-point-jump-audited}, and set
\begin{equation}\label{eq:LambdaP-audited}
 L^{\pm}(w)
 =\frac1{2\pi i}\sum_{j=0}^{N-1}\overline{\tau_j}\,l_j^{\pm}(w),
 \qquad
 \psi=L^{+}\circ g,
 \qquad
 v=L^{-}\circ f.
\end{equation}
The sum is finite, so $\psi\in\Bplus$ and $v\in\Bminus$. Summing the jump relations gives, for almost every $\zeta\in\T$,
\[
 \psi(h^{-1}(\zeta))-v(\zeta)=u(\zeta).
\]
Applying $P_{+}$ yields the projected composition equation
\begin{equation}\label{eq:projection-equation-audited}
 A_{h^{-1}}\psi=P_{+}u.
\end{equation}
Consequently, the lower bound \eqref{eq:lower-Ahinv} gives the partition-independent estimate
\begin{equation}\label{eq:uniform-Psi-audited}
 \|\psi\|_{\BMOA}\le Cc^{-1}.
\end{equation}

 By \eqref{eq:two-point-derivative-audited},
\begin{align}
 (L^{+})'(w)
 &=\frac1{2\pi i}\sum_{j=0}^{N-1}
 \overline{\tau_j}\,
 \frac{z_{j+1}-z_j}{(w-z_{j+1})(w-z_j)}\notag\\
 &=\frac1{2\pi i}\sum_{j=0}^{N-1}
 \frac{|z_{j+1}-z_j|}{(w-z_{j+1})(w-z_j)}.
 \label{eq:positive-derivative-audited}
\end{align}

\subsection{Rectifiability}\label{sec:rectifiability-discrete}
We use a single observation point to control polygonal variation on a finite cover, thereby creating arclength.
\begin{lemma}[Fixed-point cone estimate]\label{lem:fixed-cone-audited}
Fix $w_0\in\Omega_{+}$ and write
\[
 d_0=\dist(w_0,\Gamma),\qquad
 R_0=\max_{z\in\Gamma}|w_0-z|.
\]
There are $\varepsilon_0>0$ and $c_0>0$, depending only on $d_0/R_0$, such that whenever a connected subarc $E\subset\Gamma$ satisfies
\[
 \diam E\le\varepsilon_0d_0,
\]
there is a unimodular number $\lambda_E$ for which
\begin{equation}\label{eq:fixed-cone-audited}
 \operatorname{Re}\frac{\lambda_E}{(w_0-z')(w_0-z)}
 \ge\frac{c_0}{R_0^2}
\end{equation}
for all $z,z'\in E$.
\end{lemma}

\begin{proof}
Choose $z_E\in E$ and put $v=w_0-z_E$. Then $d_0\le|v|\le R_0$, and
\[
 |(w_0-z)-v|\le\diam E\le\varepsilon_0d_0\le\varepsilon_0|v|,
 \qquad z\in E.
\]
Take $\varepsilon_0<1/20$. All vectors $w_0-z$, $z\in E$, then lie in a cone of aperture less than $\pi/12$ about $v$, and their moduli lie between $(1-\varepsilon_0)|v|$ and $(1+\varepsilon_0)|v|$. Their pairwise products lie in a cone of aperture less than $\pi/6$ about $v^2$. With
\[
 \lambda_E=\frac{v^2}{|v|^2},
\]
we obtain
\[
 \operatorname{Re}\frac{\lambda_E}{(w_0-z')(w_0-z)}
 \ge\frac{\cos(\pi/6)}{(1+\varepsilon_0)^2|v|^2}
 \ge\frac{c_0}{R_0^2}.
\]
\end{proof}

\begin{theorem}[Rectifiability from the lower bound]\label{thm:rectifiability-discrete}
If \eqref{eq:lower-Ahinv} holds, then $f|_{\T}$ has bounded variation and $\Gamma$ is rectifiable.
\end{theorem}

\begin{proof}
Fix $z_0\in\D$ and put $w_0=g(z_0)$. Let $d_0$ and $R_0$ be as in Lemma~\ref{lem:fixed-cone-audited}. By uniform continuity of $f|_{\T}$, there are finitely many closed arcs $I_1,\ldots,I_M$ whose interiors cover $\T$ and such that
\[
 \diam f(I_k)\le\varepsilon_0d_0,
 \qquad 1\le k\le M.
\]
Fix $k$ and an arbitrary ordered partition of $I_k$. In the notation of Section~\ref{sec:discrete-jumps}, put
\[
 L=\sum_j|z_{j+1}-z_j|.
\]
The cone lemma gives a unimodular number $\lambda_k$ for $f(I_k)$. Multiplying \eqref{eq:positive-derivative-audited} by $2\pi i\lambda_k$ and taking real parts yields
\[
 |(L^{+})'(w_0)|
 \ge \frac{c_0}{2\pi R_0^2}\sum_j|z_{j+1}-z_j|
 =\frac{c_0}{2\pi R_0^2}L.
\]
On the other hand, $\psi=L^{+}\circ g$, and hence
\[
 \psi'(z_0)=(L^{+})'(w_0)g'(z_0).
\]
The Bloch estimate \eqref{eq:Bloch-BMOA-audited} and the uniform bound \eqref{eq:uniform-Psi-audited} give
\[
 |(L^{+})'(w_0)|
 \le\frac{C}{(1-|z_0|^2)|g'(z_0)|}\,c^{-1}.
\]
Therefore
\[
 \sum_j|z_{j+1}-z_j|\le C_k,
\]
where $C_k$ is independent of the partition. Thus, $f\vert{}_{I_k}$ is of bounded variation, which implies that $f\vert{}_{\mathbb{T}}$ has finite total variation. Since $f$ is continuous and injective on $\mathbb{T}$, its image $\Gamma$ is a rectifiable Jordan curve, and its total variation is equal to the length of $\Gamma$.
\end{proof}

\section{Localization and the chord-arc estimate}\label{sec:alternative-whitney}

The fixed-point argument above detects finite total variation but is not scale invariant. We now place the observation region at the scale of an arbitrary boundary subarc. A comparison ball in $\Omega_{+}$ replaces the fixed point, while the Bloch estimate is replaced by the local Dirichlet-energy estimate in Lemma~\ref{lem:hyperbolic-energy-audited}. The resulting bound is
\[
 \length g(J)\lesssim\diam g(J),
\]
uniformly over all sufficiently short arcs $J\subset\T$.

\subsection{Comparison balls in the quasidisk}

We use the standard fact that every quasidisk is a uniform domain \cite{GehringOsgood}. Recall that a domain $\Omega$ is $C_u$-uniform if every pair $x,y\in\Omega$ can be joined by a rectifiable curve $\gamma\subset\Omega$ such that
\[
 \length\gamma\le C_u|x-y|
\]
and
\[
 \dist(\omega,\partial\Omega)
 \ge C_u^{-1}\min\{\length\gamma[x,\omega],
                    \length\gamma[\omega,y]\},
 \qquad \omega\in\gamma.
\]

\begin{lemma}[Interior corkscrew]\label{lem:corkscrew-audited}
There exist constants $c_\Gamma\in(0,1)$ and $r_\Gamma>0$ such that, for every $\xi\in\Gamma$ and every $0<r<r_\Gamma$, one can find $w\in\Omega_{+}$ satisfying
\begin{equation}\label{eq:corkscrew-audited}
 |w-\xi|<r,
 \qquad
 \dist(w,\Gamma)\ge c_\Gamma r.
\end{equation}
\end{lemma}

\begin{proof}
Fix $x_*\in\Omega_{+}$ and set $\delta_*=\dist(x_*,\Gamma)$. Choose
\[
 r_\Gamma<\delta_*/4.
\]
Given $\xi\in\Gamma$ and $0<r<r_\Gamma$, select $x\in\Omega_{+}$ with $|x-\xi|<r/8$. Let $\gamma\subset\Omega_{+}$ be a $C_u$-uniform curve joining $x$ to $x_*$. Since
\[
 |x_*-\xi|\ge\delta_*>4r,
\]
there is a first point $w\in\gamma$, measured from $x$, for which $|w-\xi|=r/2$. The initial portion of $\gamma$ has length at least
\[
 |w-x|\ge |w-\xi|-|x-\xi|>3r/8,
\]
whereas the terminal portion has length at least
\[
 |x_*-w|\ge |x_*-\xi|-|w-\xi|
 >\delta_*-r/2>r.
\]
The uniform-domain condition therefore gives
\[
 \dist(w,\Gamma)\ge \frac{3r}{8C_u}.
\]
Thus \eqref{eq:corkscrew-audited} holds with $c_\Gamma=3/(8C_u)$.
\end{proof}

\begin{lemma}[Pull-back of an interior Whitney ball]\label{lem:pullback-ball-audited}
Let $w\in\Omega_{+}$, put
\[
 z=g^{-1}(w),
 \qquad
 \delta=\dist(w,\Gamma),
\]
and fix $0<\varepsilon<1/4$. Then
\begin{equation}\label{eq:pullback-hyperbolic-ball}
 g^{-1}\bigl(B(w,\varepsilon\delta)\bigr)
 \subset B_{\D}^{\mathrm{hyp}}(z,R_\varepsilon),
 \qquad
 R_\varepsilon=\frac{2\varepsilon}{1-\varepsilon}.
\end{equation}
Consequently, if $\Lambda$ is holomorphic in $\Omega_{+}$ and $\Lambda\circ g\in\BMOA(\D)$, then
\begin{equation}\label{eq:physical-energy-audited}
 \int_{B(w,\varepsilon\delta)}
 |\Lambda'(\omega)|^2\,dA(\omega)
 \le C_\varepsilon\|\Lambda\circ g\|_{\BMOA}^{\,2}.
\end{equation}
\end{lemma}

\begin{proof}
Fix $\omega\in B(w,\varepsilon\delta)$. The segment
\[
 [w,\omega]\subset B(w,\varepsilon\delta)\subset\Omega_{+}
\]
and every point $\eta$ of this segment satisfies
\[
 \dist(\eta,\Gamma)
 \ge\delta-|\eta-w|
 >(1-\varepsilon)\delta.
\]
With the normalization
\[
 \lambda_{\Omega_{+}}(\eta)
 \le\frac{2}{\dist(\eta,\Gamma)}
\]
for the hyperbolic density, the hyperbolic length of the segment is bounded by
\[
 \int_{[w,\omega]}\lambda_{\Omega_{+}}(\eta)\,|d\eta|
 \le\frac{2|\omega-w|}{(1-\varepsilon)\delta}
 <\frac{2\varepsilon}{1-\varepsilon}=R_\varepsilon.
\]
Conformal invariance of hyperbolic distance proves
\eqref{eq:pullback-hyperbolic-ball}.

Now set $\Phi=\Lambda\circ g$. By the chain rule and the area formula,
\begin{align*}
 \int_{B(w,\varepsilon\delta)}
 |\Lambda'(\omega)|^2\,dA(\omega)
 &=
 \int_{g^{-1}(B(w,\varepsilon\delta))}
 |\Lambda'(g(\zeta))|^2|g'(\zeta)|^2\,dA(\zeta)\\
 &=
 \int_{g^{-1}(B(w,\varepsilon\delta))}
 |\Phi'(\zeta)|^2\,dA(\zeta).
\end{align*}
The domain of integration is contained in the hyperbolic disk in
\eqref{eq:pullback-hyperbolic-ball}. Lemma~\ref{lem:hyperbolic-energy-audited}
therefore gives \eqref{eq:physical-energy-audited}.
\end{proof}

\begin{lemma}[Comparison ball and common cone]\label{lem:comparison-ball-audited}
There exist constants
\[
 A>1,\qquad
 \varepsilon\in(0,10^{-2}),\qquad
 c_1>0,\qquad
 \ell_0>0
\]
such that the following holds. For every arc $J\subset\T$ with
$|J|<\ell_0$, there are a point $w_J\in\Omega_{+}$, a disk
\[
 B_J=B(w_J,\varepsilon\delta_J),
 \qquad
 \delta_J=\dist(w_J,\Gamma),
\]
and a unimodular number $\lambda_J$ for which
\begin{equation}\label{eq:ball-scale-audited}
 c_\Gamma A\diam g(J)\le\delta_J<A\diam g(J),
 \qquad
 |B_J|\asymp\bigl(\diam g(J)\bigr)^2,
\end{equation}
and
\begin{equation}\label{eq:ball-cone-audited}
 \operatorname{Re}
 \frac{\lambda_J}
 {(w-z')(w-z)}
 \ge
 \frac{c_1}{\bigl(\diam g(J)\bigr)^2}
\end{equation}
for every $w\in B_J$ and every $z,z'\in g(J)$.
\end{lemma}

\begin{proof}
Choose $A>1$ so large that
\[
 \frac{1}{c_\Gamma A}<10^{-3},
\]
and then choose $\varepsilon>0$ so small that
\[
 \varepsilon+\frac{1}{c_\Gamma A}<10^{-2}.
\]
The boundary map $g|_{\T}$ is uniformly continuous, so
$\diam g(J)\to0$ uniformly as $|J|\to0$. Hence one may choose
$\ell_0>0$ such that
\[
 A\diam g(J)<r_\Gamma
 \qquad\text{whenever } |J|<\ell_0.
\]

Fix such an arc $J$ and choose $\xi_J\in g(J)$. Apply
Lemma~\ref{lem:corkscrew-audited} with
\[
 r=A\diam g(J).
\]
It gives $w_J\in\Omega_{+}$ such that
\[
 |w_J-\xi_J|<A\diam g(J),
 \qquad
 \delta_J\ge c_\Gamma A\diam g(J).
\]
Because $\xi_J\in\Gamma$, one also has
\[
 \delta_J\le |w_J-\xi_J|<A\diam g(J).
\]
This proves the first estimate in \eqref{eq:ball-scale-audited}. Since
\[
 |B_J|=\pi\varepsilon^2\delta_J^2,
\]
the area comparison follows as well.

Put $v_J=w_J-\xi_J$. For $w\in B_J$ and $z\in g(J)$,
\begin{align*}
 |(w-z)-v_J|
 &\le |w-w_J|+|z-\xi_J|\\
 &<\varepsilon\delta_J+\diam g(J)\\
 &\le
 \left(\varepsilon+\frac{1}{c_\Gamma A}\right)|v_J|
 <10^{-2}|v_J|.
\end{align*}
Thus every vector $w-z$, with $w\in B_J$ and $z\in g(J)$, lies in a
fixed narrow cone about $v_J$, and its modulus is comparable to
$|v_J|$. Consequently the products
\[
 (w-z')(w-z)
\]
lie in a cone of aperture less than $\pi/10$ about $v_J^2$. Taking
\[
 \lambda_J=\frac{v_J^2}{|v_J|^2}
\]
gives
\[
 \operatorname{Re}
 \frac{\lambda_J}{(w-z')(w-z)}
 \ge\frac{c}{|v_J|^2}.
\]
Finally,
\[
 |v_J|<A\diam g(J),
\]
so the last quantity is at least
\[
 \frac{c/A^2}{(\diam g(J))^2}.
\]
Absorbing the fixed factor $A^{-2}$ into $c_1$ proves
\eqref{eq:ball-cone-audited}.
\end{proof}

\subsection{Localized polygonal length}

Under \eqref{eq:lower-Ahinv}, Theorem~\ref{thm:rectifiability-discrete}
has already established finite total length. We now keep the same finite
logarithmic tests and move the observation region to the scale of the
arc. The purpose of this second step is to obtain a uniform
length--diameter bound.

\begin{theorem}[Localized polygonal-length estimate]\label{thm:local-polygon-audited}
Assume the lower estimate \eqref{eq:lower-Ahinv}. There is a constant
$C<\infty$ such that, for every arc $J\subset\T$ with $|J|<\ell_0$
and every ordered finite partition
\[
 z_0,z_1,\ldots,z_N
\]
of the physical arc $g(J)$,
\begin{equation}\label{eq:polygon-length-audited}
 \sum_{j=0}^{N-1}|z_{j+1}-z_j|
 \le C\diam g(J).
\end{equation}
Consequently,
\begin{equation}\label{eq:length-diameter-audited}
 \length g(J)\le C\diam g(J).
\end{equation}
\end{theorem}

\begin{proof}
Let $I=h(J)$ be the corresponding exterior parameter arc and use the
induced ordered partition of $I$ in the construction of
Section~\ref{sec:discrete-jumps}. Set
\[
 L=\sum_{j=0}^{N-1}|z_{j+1}-z_j|,
 \qquad
 d=\diam g(J).
\]
Let $B_J$ and $\lambda_J$ be furnished by
Lemma~\ref{lem:comparison-ball-audited}. Formula
\eqref{eq:positive-derivative-audited} gives
\[
 2\pi i\lambda_J(L^+)'(w)
 =
 \sum_{j=0}^{N-1}
 |z_{j+1}-z_j|
 \frac{\lambda_J}{(w-z_{j+1})(w-z_j)}.
\]
Every summand has nonnegative real part bounded below by
$c_1|z_{j+1}-z_j|/d^2$. Therefore, for every $w\in B_J$,
\begin{align}
 2\pi |(L^+)'(w)|
 &\ge
 \operatorname{Re}\!\left(2\pi i\lambda_J(L^+)'(w)\right)\notag\\
 &\ge \frac{c_1}{d^2}
 \sum_{j=0}^{N-1}|z_{j+1}-z_j|
 =\frac{c_1L}{d^2}.
 \label{eq:pointwise-lower-audited}
\end{align}

Recall that $\psi=L^+\circ g$. By the area formula and
Lemma~\ref{lem:pullback-ball-audited},
\begin{align}
 \int_{B_J}|(L^+)'(w)|^2\,dA(w)
 &=
 \int_{g^{-1}(B_J)}|\psi'(\zeta)|^2\,dA(\zeta)\notag\\
 &\le C_\varepsilon\|\psi\|_{\BMOA}^{\,2}.
 \label{eq:localized-energy-upper}
\end{align}
On the other hand, \eqref{eq:pointwise-lower-audited} and
\eqref{eq:ball-scale-audited} imply
\begin{align*}
 \int_{B_J}|(L^+)'(w)|^2\,dA(w)
 &\ge
 \frac{c_1^2L^2}{4\pi^2d^4}|B_J|\\
 &\ge c\,\frac{L^2}{d^2}.
\end{align*}
Thus, before using the projection lower bound, we have the discrete
test estimate
\begin{equation}\label{eq:discrete-test-norm}
 \frac{L}{d}\le C\|\psi\|_{\BMOA}.
\end{equation}
Now apply \eqref{eq:uniform-Psi-audited} to obtain
\[
 L\le Cc^{-1}d,
\]
uniformly in the partition. This proves
\eqref{eq:polygon-length-audited}. Since rectifiability has already
been proved in Theorem~\ref{thm:rectifiability-discrete}, taking the
supremum over all ordered finite partitions gives the arclength
estimate \eqref{eq:length-diameter-audited}.
\end{proof}

\begin{theorem}\label{thm:geometric-converse-audited}
If \eqref{eq:lower-Ahinv} holds, then $\Gamma$ is chord-arc.
\end{theorem}

\begin{proof}
Theorem~\ref{thm:rectifiability-discrete} gives
$\length\Gamma<\infty$. Since $\Gamma$ is a quasicircle, it has the
bounded-turning property: there is $C_q\ge1$ such that, for any
$z_1,z_2\in\Gamma$, one of the two subarcs $E$ joining them satisfies
\[
 \diam E\le C_q|z_1-z_2|.
\]
By uniform continuity of $g^{-1}:\Gamma\to\T$, there is $\rho>0$ such
that every connected boundary subarc of diameter less than $\rho$
has a parameter preimage of length less than $\ell_0$.

Suppose first that
\[
 |z_1-z_2|<\rho/C_q.
\]
Choose the bounded-turning subarc $E$. Then $\diam E<\rho$, so
$E=g(J)$ for an arc $J$ with $|J|<\ell_0$. By
Theorem~\ref{thm:local-polygon-audited},
\[
 \length E\le C\diam E
 \le CC_q|z_1-z_2|.
\]
The shorter-arclength subarc between $z_1$ and $z_2$ has no greater
length than $E$.

If $|z_1-z_2|\ge\rho/C_q$, the shorter-arclength subarc has length at
most $\length\Gamma/2$, and hence
\[
 \length_{\mathrm{short}}(z_1,z_2)
 \le \frac{C_q\length\Gamma}{2\rho}|z_1-z_2|.
\]
Combining the two cases proves the chord-arc inequality.
\end{proof}

\begin{proof}[Proof of Theorem~\ref{thm:main-characterization}]
If $\Gamma$ is chord-arc, Theorem~\ref{thm:chordarc-implies-projection-direct}
shows that both $A_h$ and $A_{h^{-1}}$ are bounded isomorphisms.
Proposition~\ref{prop:inverse-symmetry} gives the equivalence of their
invertibility.

Conversely, the lower estimate \eqref{eq:lower-Ahinv} first gives the
partition-independent logarithmic bound
\eqref{eq:uniform-Psi-audited}. Theorem~\ref{thm:rectifiability-discrete}
uses it at one interior point to obtain finite total length.
Theorem~\ref{thm:local-polygon-audited} then localizes the same family
to comparison balls and yields the scale-invariant estimate
\[
 \length g(J)\le C\diam g(J)
\]
on all sufficiently short conformal boundary arcs. Finally,
Theorem~\ref{thm:geometric-converse-audited} converts this local
estimate into the global chord-arc condition. Thus a lower bound for
$A_{h^{-1}}$ already forces $\Gamma$ to be chord-arc, and the direct
implication then makes both analytic projections invertible.
\end{proof}

\subsection{The Fredholm strengthening}\label{sec:fredholm}
The localized logarithmic tests also detect failure of the chord-arc
condition modulo any finite-dimensional obstruction. We
continue to work on $X=\BMOA(\D)/\C$, and use representatives
vanishing at $0$ when discussing locally uniform convergence.

\begin{lemma}[Localized approximate null vectors]\label{lem:localized-null}
If $\Gamma$ is not chord-arc, there are functions $\varphi_n\in\BMOA(\D)$
with $\varphi_n(0)=0$ such that
\begin{equation}\label{eq:localized-null-sequence}
 \|\varphi_n\|_X=1,\qquad
 \|A_{h^{-1}}\varphi_n\|_X\longrightarrow0,\qquad
 \varphi_n\longrightarrow0
 \quad\text{locally uniformly in }\D.
\end{equation}
\end{lemma}

\begin{proof}
For every $n\ge1$, there is an arc $J_n\subset\T$ with
$|J_n|<\min\{\ell_0,1/n\}$ and an ordered finite partition
$z_0^{(n)},\ldots,z_{N_n}^{(n)}$ of $E_n=g(J_n)$ such that, writing
\[
 d_n=\diam E_n,\qquad
 V_n=\sum_{j=0}^{N_n-1}|z_{j+1}^{(n)}-z_j^{(n)}|,
\]
one has $V_n/d_n>n$. Indeed, suppose this failed for some $n$ and
put $\ell=\min\{\ell_0,1/n\}$. Every finite partition of every arc
$g(J)$ with $|J|<\ell$ would then have polygonal variation at most
$n\diam g(J)$. Subdivide $\T$ into finitely many such parameter
arcs. Taking the supremum over their finite partitions gives bounded
variation on each arc; inserting the subdivision endpoints into any
partition of $\T$ then gives finite total variation. Thus $\Gamma$
would first be rectifiable. We could now take the same suprema to
obtain
\[
 \length g(J)\le n\diam g(J),\qquad |J|<\ell.
\]
The bounded-turning and small/large-chord argument in the proof of
Theorem~\ref{thm:geometric-converse-audited}, applied with this $\ell$,
would imply the chord-arc condition. That final geometric argument
uses only finite total length and a uniform small-arc
length--diameter estimate, not the projection lower bound. This is
a contradiction. Notice that no arclength is used until finite
variation has been established. Uniform continuity of $g|_\T$ gives
$d_n\to0$.

Use the induced partition of $I_n=h(J_n)$ in
Section~\ref{sec:discrete-jumps}, and denote the resulting functions by
$L_n^+$, $u_n$, and $\psi_n=L_n^+\circ g$. The projection equation
\eqref{eq:projection-equation-audited} gives
\begin{equation}\label{eq:null-output-bound}
 \|A_{h^{-1}}\psi_n\|_X=\|P_+u_n\|_X\le C.
\end{equation}
Applying the discrete estimate \eqref{eq:discrete-test-norm} to
$L_n^+$ gives
\begin{equation}\label{eq:null-input-lower}
 a_n:=\|\psi_n\|_X\ge c\frac{V_n}{d_n}\longrightarrow\infty.
\end{equation}
This estimate does not assume any lower bound for $A_{h^{-1}}$.
Set
\[
 \varphi_n=\frac{\psi_n-\psi_n(0)}{a_n}.
\]
Then $\|\varphi_n\|_X=1$ and
$\|A_{h^{-1}}\varphi_n\|_X\le C/a_n\to0$; subtraction of the
constant has no effect on the quotient equation.

It remains to prove local uniform convergence. Fix $0<r<1$ and put
\[
 \delta_r=\dist\bigl(g(\overline{r\D}),\Gamma\bigr)>0,
 \qquad M_r=\sup_{|z|\le r}|g'(z)|<\infty.
\]
The explicit derivative formula \eqref{eq:positive-derivative-audited}
and the chain rule give
\[
 \sup_{|z|\le r}|\psi_n'(z)|
 \le\frac{M_r}{2\pi\delta_r^2}V_n.
\]
After division by \eqref{eq:null-input-lower},
\[
 \sup_{|z|\le r}|\varphi_n'(z)|\le C_r d_n\longrightarrow0.
\]
Since $\varphi_n(0)=0$, integration on radial segments proves local
uniform convergence to zero. This establishes all of
\eqref{eq:localized-null-sequence}.
\end{proof}

\begin{theorem}[Upper semi-Fredholm rigidity]\label{thm:upper-semi-fredholm}
If $h\in\SQS(\T)$ and either $A_h$ or $A_{h^{-1}}$ has closed range
and a finite-dimensional kernel, then $\Gamma$ is chord-arc and both
operators are bounded isomorphisms.
\end{theorem}

\begin{proof}
By Proposition~\ref{prop:inverse-symmetry}, it suffices to consider
$A=A_{h^{-1}}$. Suppose its range is closed and its kernel $N$ is
finite-dimensional. The induced bounded bijection
$X/N\to\Ran A$ is an isomorphism of Banach spaces. Consequently,
\begin{equation}\label{eq:closed-range-distance}
 \dist_X(\varphi,N)\le C\|A\varphi\|_X,
 \qquad\varphi\in X.
\end{equation}

If $\Gamma$ were not chord-arc, take the sequence in
Lemma~\ref{lem:localized-null}. Choose $k_n\in N$ so that
\[
 \|\varphi_n-k_n\|_X
 \le\dist_X(\varphi_n,N)+1/n\longrightarrow0.
\]
The sequence $k_n$ is bounded, and $\|k_n\|_X\to1$. Since $N$ is
finite-dimensional, a subsequence converges in $X$ to $k\in N$ with
$\|k\|_X=1$. Normalize all representatives at $0$. The Bloch estimate
\eqref{eq:Bloch-BMOA-audited}, integrated on radial segments, shows
that convergence in $X$ of these normalized representatives implies
locally uniform convergence. Hence the same subsequence of
$\varphi_n$ converges locally uniformly to $k$. But
\eqref{eq:localized-null-sequence} forces $k=0$, a contradiction.

Thus $\Gamma$ is chord-arc. Theorem~\ref{thm:chordarc-implies-projection-direct}
then gives bounded invertibility of both analytic projections.
\end{proof}

\begin{proof}[Proof of Theorem~\ref{cor:fredholm-intro}]
For a chord-arc curve both operators are bounded isomorphisms by
Theorem~\ref{thm:main-characterization}, and hence are Fredholm.
Fredholmness implies closed range and a finite-dimensional kernel.
Theorem~\ref{thm:upper-semi-fredholm} gives the converse, including
the stronger upper semi-Fredholm assertion.
\end{proof}

The argument uses only the local uniform vanishing of the normalized
tests and compactness in a finite-dimensional space; no weak convergence
of the tests in BMOA is asserted or needed. In particular, the conclusion
does not say that every non-chord-arc welding has an infinite-dimensional
kernel: failure of closed range is another possible obstruction.

\section{The Faber operator and the inverse analytic projection}\label{sec:classical-faber}

The main geometric characterization has already been proved by the discrete logarithmic method.  We now turn to the classical Faber operator .  
 Initially the Faber operator was investigated acting on
the disc algebra in the literature (see \cite{Anderson, AndersonClunie, KovariPommerenke}). Anderson \cite{Anderson} conjectured that the Faber operator acts as a bounded isomorphism on spaces of analytic functions, such as Besov spaces. However, this turns out to be false even for the Dirichlet space (see \cite{WeiWangHu}).

 This part has two main purposes. One is to identify the Faber integral operator with the inverse analytic projection, which subsequently allows us to disprove Anderson's conjecture in the setting of BMOA. The other is to give a continuous chord-arc test using the interior common cone.

 We begin with this geometric converse. Throughout this section, $\Gamma$ is a bounded rectifiable quasicircle. Its welding $h$ is quasisymmetric, but strong quasisymmetry is not imposed as a standing hypothesis. It will be assumed explicitly when the BMO pull-back is used, and it will follow from boundedness of the Faber operator by Theorem~\ref{thm:bounded-faber-implies-chord-arc}. The boundary maps of $f$ and $g$ are absolutely continuous, and
\[
 f'\in H^1(\Ds),\qquad g'\in H^1(\D).
\]
For $w\in\Omega_{+}$, the function
\[
 K_w(\zeta)=\frac{f'(\zeta)}{f(\zeta)-w}
\]
belongs to $H^1(\Ds)$, since $f(\Ds)=\Omega_{-}$ and the denominator stays uniformly away from zero. Hence \eqref{eq:Faber-intro-new} is well-defined for every $\varphi\in\BMOA(\D)$.  For constant data the contour integral gives $\Faber_f^+1=1$, and hence the operator descends to the quotient by constants.

For an analytic polynomial $p$, the same contour integral is classical for every $w\notin\Gamma$; denote its restrictions to $\Omega_{+}$ and $\Omega_{-}$ by $\Faber_f^{+}p$ and $\Faber_f^{-}p$, respectively. Moreover, $\Faber_f^{-}p(\infty)=0$. The classical Plemelj relation gives
\begin{equation}\label{eq:Faber-jump}
 (\Faber_f^{+}p)^*-(\Faber_f^{-}p)^*=p\circ f^{-1}
 \qquad\text{a.e. on }\Gamma.
\end{equation}

The BMOA Faber problem asks when $\Faber_f^{+}$ is a bounded isomorphism; see \cite{LiuShenFaber,WeiWangHu}. We identify this question exactly with invertibility of the analytic projection.

\subsection{A continuous arclength test and the interior proof}\label{sec:tangent-test}
The finite logarithmic tests above did not require arclength to exist.
Under the standing rectifiability assumption, the comparison-ball
method now admits a continuous version: the chord directions are
replaced by the almost-everywhere unit tangent, discrete chord-length weights \(|z_{j+1}-z_j|\) are replaced by the arclength measure \(|\gamma'(t)|\,dt\), and the discrete denominators are replaced by \((w-\gamma(t))^2\). We construct this test directly and then apply the geometric
common cone to its derivative. Only rectifiability and the quasidisk geometry are used in this subsection.

Assume that
\begin{equation}\label{eq:Faber-bounded}
 \Faber_f^{+}:\BMOA(\D)\longrightarrow\BMOA(\Omega_{+})
\end{equation}
is bounded. We prove that $\Gamma$ is chord-arc.

Because $\Gamma$ is rectifiable, the boundary parametrization
\[
 \gamma(t)=f(e^{it})
\]
is absolutely continuous.  Define the unit tangent
\[
 \tau(t)=\begin{cases}
 \gamma'(t)/|\gamma'(t)|,&\gamma'(t)\ne0,\\
 1,&\gamma'(t)=0.
 \end{cases}
\]
Then $|\tau|=1$ and
\begin{equation}\label{eq:tangent-line-element}
 \overline{\tau(t)}\,\gamma'(t)=|\gamma'(t)|
 \quad\text{for a.e. }t.
\end{equation}

Fix an arc $I\subset\T$ and define
\begin{equation}\label{eq:uI}
 u_I(e^{it})=\overline{\tau(t)}\chi_I(e^{it}),
 \qquad
 \varphi_I=P_{+}u_I.
\end{equation}
Since $|u_I|\le1$, one has $\|u_I\|_{\BMO}\le2$ and $|(u_I)_\T|\le1$.  The boundedness of the Riesz projection therefore gives a uniform estimate in any standard full BMOA norm:
\begin{equation}\label{eq:phiI-bound}
 |\varphi_I(0)|+\|\varphi_I\|_{\BMOA}
 \le C_P.
\end{equation}
Let
\[
 F_I=\Faber_f^{+}\varphi_I,
 \qquad
 \Psi_I=F_I\circ g.
\]
Then
\begin{equation}\label{eq:PsiI-bound}
 \|\Psi_I\|_{\BMOA}
 \le C_P\|\Faber_f^{+}\|.
\end{equation}
\begin{lemma}\label{lem:spectral-replacement}
For every $w\in\Omega_{+}$,
\begin{equation}\label{eq:replace-Pplus}
 F_I(w)=\frac{1}{2\pi i}\int_{\T}
 \frac{f'(\zeta)u_I(\zeta)}{f(\zeta)-w}\,d\zeta,
\end{equation}
where the integral is understood as the $H^1$--BMO pairing.
\end{lemma}

\begin{proof}
Put
\[
 K_w(\zeta)=\frac{f'(\zeta)}{f(\zeta)-w}.
\]
The function $K_w$ is holomorphic in $\Ds$, belongs to $H^1(\Ds)$, and vanishes at infinity.  Its Laurent expansion is
\[
 K_w(\zeta)=\sum_{m=1}^{\infty}a_m\zeta^{-m};
\]
in particular, its boundary spectrum consists only of the frequencies $-1,-2,\ldots$.

Let
\[
 q_I=P_{-}u_I=u_I-P_{+}u_I.
\]
Then $q_I\in\BMO$ and its Fourier spectrum also consists only of the strictly negative frequencies.  We claim that
\begin{equation}\label{eq:negative-pairing-zero}
 \frac1{2\pi i}\int_\T K_w(\zeta)q_I(\zeta)\,d\zeta=0.
\end{equation}
To justify this at the endpoint level, let $K_{w,r}(\zeta)=K_w(r\zeta)$ for $r>1$.  Then $K_{w,r}$ has an absolutely convergent Laurent series on $\T$, contains only strictly negative frequencies, and $K_{w,r}\to K_w$ in $H^1(\Ds)$ as $r\downarrow1$.  Let $\sigma_Nq_I$ be the $N$th Fejer mean of $q_I$.  It is a trigonometric polynomial containing only strictly negative frequencies, and $\sigma_Nq_I\to q_I$ weak-star in BMO.  For fixed $r>1$ and $N$, the product $K_{w,r}\sigma_Nq_I$ contains only powers $\zeta^{-k}$ with $k\ge2$; hence its $\zeta^{-1}$ coefficient is zero and
\[
 \int_\T K_{w,r}(\zeta)\sigma_Nq_I(\zeta)\,d\zeta=0.
\]
First let $N\to\infty$, using the BMO--$H^1$ pairing, and then let $r\downarrow1$, using $H^1$ convergence.  This proves \eqref{eq:negative-pairing-zero}.

Since $P_{+}u_I=u_I-q_I$, the defining Faber pairing gives
\begin{align*}
 F_I(w)
 &=\frac1{2\pi i}\int_\T K_w(\zeta)P_{+}u_I(\zeta)\,d\zeta\\
 &=\frac1{2\pi i}\int_\T K_w(\zeta)u_I(\zeta)\,d\zeta,
\end{align*}
which is \eqref{eq:replace-Pplus}.
\end{proof}

Parametrize $\zeta=e^{is}$.  By \eqref{eq:tangent-line-element},
\[
 \overline{\tau(s)}f'(e^{is})ie^{is}\,ds
 =\overline{\tau(s)}\gamma'(s)\,ds
 =|\gamma'(s)|\,ds.
\]
Therefore
\begin{equation}\label{eq:positive-arclength-kernel}
 F_I(w)=\frac{1}{2\pi i}\int_I
 \frac{|\gamma'(s)|}{\gamma(s)-w}\,ds.
\end{equation}
This formula has an ordinary, absolutely convergent meaning whenever $w$ stays a positive distance from the compact arc $\gamma(I)$.

\begin{theorem}[Interior proof of the Faber converse]\label{thm:bounded-faber-implies-chord-arc}
Let $\Gamma$ be a bounded rectifiable quasicircle. If $\Faber_f^{+}$ is bounded from $\BMOA(\D)$ to
$\BMOA(\Omega_{+})$, then $\Gamma$ is chord-arc.
\end{theorem}

\begin{proof}
Fix $J\subset\T$ with $|J|<\ell_0$ as in
Lemma~\ref{lem:comparison-ball-audited}, and set
\[
 E=g(J),\qquad I=h(J),\qquad d=\diam E,\qquad
 L=\int_I|\gamma'(t)|\,dt=\length E.
\]
The identity $I=h(J)$ is used only to specify the same physical arc
$f(I)=g(J)$; no estimate for a welding pull-back or change of variables
through $h$ is used. Use the function $F_I$ defined above. By
\eqref{eq:positive-arclength-kernel},
\[
 2\pi iF_I(w)=\int_I\frac{|\gamma'(t)|}{\gamma(t)-w}\,dt.
\]
The arc $E$ is rectifiable. On any compact subset of $\Omega_+$ the
denominator stays a positive distance from zero, so the integrand
and its derivative are dominated by constant multiples of
$|\gamma'(t)|\in L^1(I)$. Differentiation under the integral is
therefore justified and gives
\begin{equation}\label{eq:continuous-derivative}
 2\pi iF_I'(w)=\int_I
       \frac{|\gamma'(t)|}{(w-\gamma(t))^2}\,dt.
\end{equation}

Let $B_J$ and $\lambda_J$ be furnished by
Lemma~\ref{lem:comparison-ball-audited}. Taking
$z=z'=\gamma(t)$ in \eqref{eq:ball-cone-audited} gives
\[
 \operatorname{Re}\frac{\lambda_J}{(w-\gamma(t))^2}
 \ge\frac{c_1}{d^2},\qquad w\in B_J,\quad e^{it}\in I.
\]
The weight $|\gamma'(t)|$ is nonnegative. Thus
\eqref{eq:continuous-derivative} yields the pointwise lower bound
\begin{align}
 2\pi|F_I'(w)|
 &\ge\operatorname{Re}\bigl(2\pi i\lambda_J F_I'(w)\bigr)\notag\\
 &=\int_I |\gamma'(t)|\,
       \operatorname{Re}\frac{\lambda_J}{(w-\gamma(t))^2}\,dt\notag\\
 &\ge\frac{c_1L}{d^2},\qquad w\in B_J.
 \label{eq:continuous-pointwise-lower}
\end{align}
Since $|B_J|\asymp d^2$ by \eqref{eq:ball-scale-audited},
\begin{equation}\label{eq:continuous-energy-lower}
 \int_{B_J}|F_I'(w)|^2\,dA(w)
 \ge\frac{c_1^2L^2}{4\pi^2d^4}|B_J|
 \ge c\frac{L^2}{d^2}.
\end{equation}
On the other hand, the local energy estimate in
Lemma~\ref{lem:pullback-ball-audited}, followed by the Faber bound
\eqref{eq:PsiI-bound}, gives
\begin{align}
 \int_{B_J}|F_I'(w)|^2\,dA(w)
 &\le C_\varepsilon\|F_I\circ g\|_{\BMOA}^{\,2}\notag\\
 &\le C\|\Faber_f^+\|^2.
 \label{eq:continuous-energy-upper}
\end{align}
Comparing \eqref{eq:continuous-energy-lower} and
\eqref{eq:continuous-energy-upper}, we obtain
\begin{equation}\label{eq:interior-faber-length}
 \length g(J)\le C\|\Faber_f^+\|\diam g(J),
 \qquad |J|<\ell_0,
\end{equation}
with a constant independent of $J$. The geometric lemmas used here require only a quasidisk; the energy lemma requires only a conformal pull-back.

Rectifiability is a standing assumption here. Consequently the final
geometric argument in the proof of
Theorem~\ref{thm:geometric-converse-audited} applies directly to
\eqref{eq:interior-faber-length}: bounded turning supplies the joining
subarc for small chords, and $\length\Gamma/2$ controls the shorter
subarc for large chords. This proves the chord-arc condition.
\end{proof}

\subsection{Weak-star approximation}\label{sec:weak-star-faber}
Rectifiability already makes the classical Faber integral well defined on BMOA. The weak-star topology is used only to extend identities and uniform estimates from analytic polynomials, which are not norm dense in BMOA, to the whole space. We use the standard $H^1$--BMO weak-star topology; pointwise statements are made for fixed representatives, while operator identities on $\Bplus$ are understood modulo constants.

\begin{lemma}\label{lem:weak-star}
Parts \textup{(a)} and \textup{(b)} are statements about BMOA alone.
For parts \textup{(c)} and \textup{(d)}, assume $h\in\SQS(\T)$.
\begin{enumerate}[label=\textup{(\alph*)}]
\item For every $\varphi\in\BMOA(\D)$ there is a sequence of analytic polynomials $p_n$ such that
\[
 \sup_n\|p_n\|_{\BMOA}\le C\|\varphi\|_{\BMOA},
 \qquad p_n\longrightarrow\varphi
\]
locally uniformly in $\D$ and weak-star in BMOA.
\item If $\varphi_n$ is bounded in BMOA and converges locally uniformly to a holomorphic function $\varphi$, then $\varphi\in\BMOA$ and $\varphi_n\to\varphi$ weak-star in BMOA.
\item The pull-back $C_h$ on BMO and the analytic projection $A_h=P_+C_h|_{\Bplus}$ are weak-star continuous.
\item If $A_h$ is a bounded isomorphism, then $A_h^{-1}$ is weak-star continuous.
\end{enumerate}
\end{lemma}

\begin{proof}
For (a), put $\varphi_r(z)=\varphi(rz)$.  Radial dilation is uniformly bounded on BMOA, and $\varphi_r\to\varphi$ weak-star and locally uniformly as $r\uparrow1$.  For fixed $r<1$, the Taylor series of $\varphi_r$ converges uniformly on $\overline\D$; hence its partial sums converge to $\varphi_r$ in BMOA norm.  Choose $r_n\uparrow1$ and then a Taylor polynomial $p_n$ of $\varphi_{r_n}$ so that
\[
 \|p_n-\varphi_{r_n}\|_{\BMOA}<2^{-n}\|\varphi\|_{\BMOA}.
\]
If $\varphi$ is constant, take $p_n=\varphi$. Otherwise the preceding choice gives the required bounded polynomial sequence, with $p_n(0)=\varphi(0)$.

For (b), apply Fatou's lemma in the Carleson-measure characterization to the locally convergent derivatives. This first gives $\varphi\in\BMOA$ with a bound by $C\sup_n\|\varphi_n\|_{\BMOA}$. The values $\varphi_n(0)$ also converge, so the full norms of the chosen representatives are bounded. Now let $a$ be an analytic $H^1$ test function.  Choose analytic polynomials $q_m$ with $q_m(0)=a(0)$ and $q_m\to a$ in $H^1$.  Pairing against a fixed $q_m$ depends on only finitely many Taylor coefficients and therefore converges under local uniform convergence.  The uniform BMOA bound gives
\[
 |\langle\varphi_n-\varphi,a-q_m\rangle|
 \le C\sup_k\|\varphi_k\|_{\BMOA}\,\|a-q_m\|_{H^1}.
\]
First let $n\to\infty$ and then $m\to\infty$.

For (c), consider the transfer operator
\[
 \mathcal L_h a=(a\circ h^{-1})(h^{-1})',
\]
where derivatives of circle homeomorphisms mean angular derivatives
of their increasing lifts. We spell out its boundedness on real $H^1$.
Since $h^{-1}\in\SQS(\T)$, the weight $w=(h^{-1})'$ belongs to
$A_\infty$ and satisfies a reverse H\"older inequality for some $p>1$;
see \cite{CF}. If $a$ is a mean-zero $L^\infty$ atom supported on an
arc $I$, with $\|a\|_\infty\le |I|^{-1}$, then $\mathcal L_h a$
is supported on $h(I)$, has integral zero, and
\[
 \|\mathcal L_h a\|_{L^p}
 \le |I|^{-1}\left(\int_{h(I)}w^p\right)^{1/p}
 \le C|h(I)|^{1/p-1},
\]
because $\int_{h(I)}w=|I|$. The $L^p$-atomic characterization of real
$H^1$ therefore gives a uniform $H^1$ bound for $\mathcal L_h a$.
On the circle the constant atom is handled by $w\in L^p(\T)$.
The atomic decomposition extends $\mathcal L_h$ to a bounded
operator on real $H^1$; see \cite{Garnett} for these Hardy-space facts.

For an atom $a$ and $u\in\BMO$, change of variables gives
\[
 \int_\T u(h(\zeta))a(\zeta)\,|d\zeta|
 =\int_\T u(\eta)(a\circ h^{-1})(\eta)(h^{-1})'(\eta)\,|d\eta|.
\]
These integrals are legitimate: on the left $u\circ h\in\BMO$ is
locally integrable, and on the right $\mathcal L_h a\in L^p$ pairs
with the local $L^{p/(p-1)}$ representative of $u$. Atomic density
and BMO--$H^1$ duality extend the identity to the full pairing. Thus
$C_h=\mathcal L_h^*$ is weak-star continuous. The Riesz projections
are adjoints of bounded projections on the real Hardy predual, so
$A_h=P_+C_h|_{\Bplus}$ is weak-star continuous as well. The same
argument applies on the quotient by constants, whose predual
consists of the mean-zero Hardy functions.

For (d), write $A_h=L^*$ for its bounded preadjoint $L$ on the analytic $H^1$ predual.  If $A_h$ is bijective, then $L^*$ is bijective.  The closed range theorem implies that $L$ has closed range; $\ker L^*=\{0\}$ implies that $\operatorname{Ran}L$ is dense, hence $L$ is onto; and $\operatorname{Ran}L^*=H^1{}^*$ implies $\ker L=\{0\}$.  Therefore $L$ is a bounded isomorphism and
\[
 A_h^{-1}=(L^{-1})^*,
\]
which is weak-star continuous.
\end{proof}

\subsection{The exact operator identity}\label{sec:faber-equivalence}

Throughout this subsection we assume, in addition to rectifiability,
that $h\in\SQS(\T)$. This is automatic when $\Gamma$ is chord-arc;
it is also a consequence of Faber boundedness by
Theorem~\ref{thm:bounded-faber-implies-chord-arc}. In particular, the
uses of $C_h$, $C_{h^{-1}}$ and Lemma~\ref{lem:weak-star}\textup{(c)--(d)}
below occur only under their stated hypotheses.

\begin{lemma}\label{lem:polynomial-faber-cauchy}
For every analytic polynomial $p$,
\[
 \Faber_fp=C_\Gamma(p\circ f^{-1})
 \qquad\text{on }\widehat\C\setminus\Gamma.
\]
Consequently,
\[
 (\Faber_f^+p)^*-(\Faber_f^-p)^*=p\circ f^{-1}
\]
for arclength almost every point of $\Gamma$.
\end{lemma}

\begin{proof}
The boundary function $p\circ f^{-1}$ is continuous on $\Gamma$.  Using the absolutely continuous boundary parametrization $\zeta=f(\eta)$ gives, for $w\notin\Gamma$,
\begin{align*}
 C_\Gamma(p\circ f^{-1})(w)
 &=\frac1{2\pi i}\int_\Gamma
 \frac{p(f^{-1}(\zeta))}{\zeta-w}\,d\zeta\\
 &=\frac1{2\pi i}\int_\T
 \frac{p(\eta)f'(\eta)}{f(\eta)-w}\,d\eta\\
 &=\Faber_fp(w).
\end{align*}
The Plemelj jump formula on a rectifiable Jordan curve now gives the boundary identity.  Since the exterior Cauchy integral vanishes at infinity, $\Faber_f^-p(\infty)=0$.
\end{proof}

\begin{lemma}\label{lem:polynomial-identity}
For every analytic polynomial $p$,
\begin{equation}\label{eq:poly-identity}
 A_{h^{-1}}\bigl(C_g\Faber_f^{+}p\bigr)=p.
\end{equation}
\end{lemma}

\begin{proof}
Let
\[
 F_{+}=\Faber_f^{+}p,
 \qquad
 F_{-}=\Faber_f^{-}p,
 \qquad
 \psi=F_{+}\circ g,
 \qquad
 v_{-}=F_{-}\circ f.
\]
Since $p$ is a polynomial, $F_{+}$ is a finite linear combination of Faber polynomials and $\psi\in \Bplus$.  Also $F_{-}$ is holomorphic in $\Omega_{-}$ and vanishes at infinity.
More explicitly, if $P=F_+$ is the corresponding Faber polynomial
combination, the classical polynomial-part identity gives
\[
 F_-(w)=P(w)-p(f^{-1}(w)),\qquad w\in\Omega_-.
\]
Indeed, $P$ is the polynomial part at infinity of $p\circ f^{-1}$;
contour deformation gives the displayed exterior branch. Both terms
extend continuously to $\Gamma$, their difference is bounded near
$\Gamma$, and it vanishes at infinity. Thus $F_-\circ f\in H_0^\infty(\Ds)$.

Compose the jump relation of Lemma~\ref{lem:polynomial-faber-cauchy} with the exterior boundary map $f$.  Using
\[
 f=g\circ h^{-1}\quad\text{on }\T,
\]
we obtain
\[
 p=F_{+}\circ f-F_{-}\circ f
 =\psi\circ h^{-1}-v_{-}.
\]
The exterior Hardy-class assertion proved above gives $v_{-}\in\Bminus$. Applying $P_{+}$ gives
\[
 p=P_{+}(\psi\circ h^{-1})=A_{h^{-1}}\psi,
\]
which is \eqref{eq:poly-identity}.
\end{proof}

\begin{proposition}[A lower projection estimate produces the bounded Faber operator]\label{prop:lower-bound-produces-faber}
Assume that $\Gamma$ is rectifiable and that, for some $c>0$,
\begin{equation}\label{eq:lower-bound-faber-section}
 \|A_{h^{-1}}\psi\|_{\BMOA}\ge c\|\psi\|_{\BMOA},\qquad \psi\in \Bplus.
\end{equation}
Then the classical interior Faber operator maps $\BMOA(\D)$ boundedly into $\BMOA(\Omega_+)$ and, with
\[
 T_f=C_g\Faber_f^+,
\]
one has
\begin{equation}\label{eq:right-inverse-from-lower-bound}
 A_{h^{-1}}T_f=I_{\Bplus}.
\end{equation}
Consequently $A_{h^{-1}}$ is surjective.  Since \eqref{eq:lower-bound-faber-section} already implies injectivity and closed range, $A_{h^{-1}}$ is a bounded isomorphism and
\begin{equation}\label{eq:inverse-from-lower-bound-faber}
 T_f=(A_{h^{-1}})^{-1}.
\end{equation}
In particular,
\[
 \|\Faber_f^+\varphi\|_{\BMOA(\Omega_+)}\le Cc^{-1}\|\varphi\|_{\BMOA(\D)}.
\]
\end{proposition}

\begin{proof}
We separate the polynomial estimate, the weak-star passage, and the right-inverse identity.

\emph{Step 1: uniform estimate for polynomial data.}
For an analytic polynomial $p$, Lemma~\ref{lem:polynomial-identity} gives
\[
 A_{h^{-1}}\bigl(C_g\Faber_f^+p\bigr)=p.
\]
The lower bound \eqref{eq:lower-bound-faber-section} therefore yields
\begin{equation}\label{eq:poly-faber-lower-estimate}
 \|C_g\Faber_f^+p\|_{\BMOA}\le c^{-1}\|p\|_{\BMOA}.
\end{equation}
Thus the desired BMOA estimate is already available on the polynomial subspace, even though that subspace is not norm dense in BMOA.

\emph{Step 2: passage from polynomials to an arbitrary BMOA function.}
Fix $\varphi\in \Bplus$.  Choose the bounded polynomial sequence $p_n$ supplied by Lemma~\ref{lem:weak-star}:
\[
 \sup_n\|p_n\|_{\BMOA}\le C\|\varphi\|_{\BMOA},\qquad
 p_n\stackrel{w^*}{\longrightarrow}\varphi,
\]
and $p_n\to\varphi$ locally uniformly in $\D$.  Put
\[
 \psi_n=C_g\Faber_f^+p_n.
\]
By \eqref{eq:poly-faber-lower-estimate},
\begin{equation}\label{eq:psi-n-bounded}
 \sup_n\|\psi_n\|_{\BMOA}\le Cc^{-1}\|\varphi\|_{\BMOA}.
\end{equation}
For every fixed $z\in\D$, set $w=g(z)\in\Omega_+$.  Rectifiability gives $f'\in H^1(\Ds)$, and
\[
 K_w(\zeta)=\frac{f'(\zeta)}{f(\zeta)-w}\in H^1(\Ds).
\]
Hence the Faber point evaluation is a weak-star continuous $H^1$--BMOA pairing:
\[
 \psi_n(z)=\Faber_f^+p_n(w)
 =\frac1{2\pi i}\langle p_n,K_w\rangle
 \longrightarrow
 \frac1{2\pi i}\langle\varphi,K_w\rangle
 =\Faber_f^+\varphi(w).
\]
Thus $\psi_n$ converges pointwise on $\D$ to the analytic function
\[
 \psi(z)=\Faber_f^+\varphi(g(z)).
\]
The uniform BMOA bound \eqref{eq:psi-n-bounded}, together with boundedness of one point value (for instance at $z=0$, which follows from the displayed pointwise convergence), makes $\{\psi_n\}$ locally uniformly bounded.  Montel's theorem therefore upgrades the pointwise convergence to local uniform convergence.  Lemma~\ref{lem:weak-star}(b) gives
\[
 \psi_n\stackrel{w^*}{\longrightarrow}\psi
 \quad\text{in BMOA}.
\]
Since the BMOA norm is weak-star lower semicontinuous on bounded sets,
\[
 \|\psi\|_{\BMOA}\le \liminf_n\|\psi_n\|_{\BMOA}
 \le Cc^{-1}\|\varphi\|_{\BMOA}.
\]
Therefore $\Faber_f^+\varphi\in\BMOA(\Omega_+)$ and the Faber operator is bounded.

\emph{Step 3: the right-inverse identity.}
For every $n$,
\[
 A_{h^{-1}}\psi_n=p_n.
\]
The operator $A_{h^{-1}}$ is weak-star continuous by Lemma~\ref{lem:weak-star}(c).  Passing to the weak-star limit on both sides gives
\[
 A_{h^{-1}}\psi=\varphi.
\]
Since $\psi=T_f\varphi$, this proves \eqref{eq:right-inverse-from-lower-bound}.  Hence $A_{h^{-1}}$ is onto.  The lower bound gives injectivity, and the bounded inverse theorem now shows that $A_{h^{-1}}$ is a bounded isomorphism.  Its inverse is the right inverse $T_f$, proving \eqref{eq:inverse-from-lower-bound-faber}.
\end{proof}

\begin{proposition}\label{prop:faber-bounded-right-inverse}
Assume that
\[
 \Faber_f^{+}:\BMOA(\D)\longrightarrow\BMOA(\Omega_{+})
\]
is bounded, and set
\[
 T_f=C_g\Faber_f^{+}:\Bplus\longrightarrow \Bplus.
\]
Then
\begin{equation}\label{eq:right-inverse}
 A_{h^{-1}}T_f=I_{\Bplus}.
\end{equation}
Consequently, $T_f$ is injective and $A_{h^{-1}}$ is surjective.
\end{proposition}

\begin{proof}
Fix $\varphi\in \Bplus$.  By Lemma~\ref{lem:weak-star}(a), choose analytic polynomials $p_n$ such that
\[
 \sup_n\|p_n\|_{\BMOA}\le C\|\varphi\|_{\BMOA},
 \qquad p_n\to\varphi
\]
weak-star and locally uniformly in $\D$.

For every $w\in\Omega_+$, the kernel
\[
 K_w(\zeta)=\frac{f'(\zeta)}{f(\zeta)-w}
\]
belongs to $H^1(\Ds)$.  Hence the Faber point evaluation is a weak-star continuous functional on BMOA:
\[
 \Faber_f^+p_n(w)
 =\frac1{2\pi i}\langle p_n,K_w\rangle
 \longrightarrow
 \frac1{2\pi i}\langle\varphi,K_w\rangle
 =\Faber_f^+\varphi(w).
\]
Since $\Faber_f^+$ is bounded, the BMOA seminorms of
\[
 T_fp_n=(\Faber_f^+p_n)\circ g
\]
are uniformly bounded.  The preceding pointwise convergence, after putting $w=g(z)$, also gives boundedness of the values $T_fp_n(0)$.  A BMOA seminorm together with one point value controls the function on every compact subset of the disk.  Thus $\{T_fp_n\}$ is locally uniformly bounded.  Since it converges pointwise to $T_f\varphi$, Montel's theorem, or Vitali's theorem, upgrades the convergence to local uniform convergence.  Lemma~\ref{lem:weak-star}(b) now gives
\[
 T_fp_n\stackrel{w^*}{\longrightarrow}T_f\varphi.
\]

For polynomials, Lemma~\ref{lem:polynomial-identity} gives
\[
 A_{h^{-1}}T_fp_n=p_n.
\]
The operator $A_{h^{-1}}$ is weak-star continuous by Lemma~\ref{lem:weak-star}(c).  Passing to the weak-star limit yields
\[
 A_{h^{-1}}T_f\varphi=\varphi.
\]
This proves \eqref{eq:right-inverse}.  A right inverse is necessarily injective, while the operator admitting it is necessarily surjective.
\end{proof}

\begin{proposition}\label{prop:projection-inverse-produces-faber}
Assume that $A_{h^{-1}}$ is a bounded isomorphism, and put
\[
 S=(A_{h^{-1}})^{-1}.
\]
Then the classical interior Faber operator is bounded and
\begin{equation}\label{eq:projection-produces-faber}
 C_g\Faber_f^+=S.
\end{equation}
In particular, $\Faber_f^+$ is a bounded isomorphism from $\BMOA(\D)$ onto $\BMOA(\Omega_+)$.
\end{proposition}

\begin{proof}
Bounded invertibility of $A_{h^{-1}}$ gives
\[
 \|A_{h^{-1}}\psi\|_{\BMOA}
 \ge\|S\|^{-1}\|\psi\|_{\BMOA}.
\]
Proposition~\ref{prop:lower-bound-produces-faber}, whose proof already
includes the polynomial approximation and the endpoint passage, therefore
shows that the classical Faber operator is bounded and that
$A_{h^{-1}}C_g\Faber_f^+=I$. Multiplication by $S$ gives
\eqref{eq:projection-produces-faber} on BMOA modulo constants.
Finally $\Faber_f^+=C_g^{-1}S$ is a composition of bounded
isomorphisms. This proves the assertion without repeating the
weak-star approximation or identifying unnormalized point values of
quotient elements.
\end{proof}

\begin{theorem}\label{thm:faber-projection-equivalence}
Assume that $\Gamma$ is rectifiable and $h\in\SQS(\T)$. Then the following are equivalent:
\begin{enumerate}[label=\textup{(\roman*)}]
\item $A_h$ is a bounded isomorphism;
\item $A_{h^{-1}}$ is a bounded isomorphism;
\item $\Faber_f^{+}:\BMOA(\D)\to\BMOA(\Omega_{+})$ is a bounded isomorphism.
\end{enumerate}
When these conditions hold,
\begin{equation}\label{eq:faber-projection-identity}
 \boxed{\quad C_g\Faber_f^{+}=(A_{h^{-1}})^{-1}.\quad}
\end{equation}
\end{theorem}

\begin{proof}
Proposition~\ref{prop:inverse-symmetry} gives (i)$\Leftrightarrow$(ii).  Proposition~\ref{prop:projection-inverse-produces-faber} gives (ii)$\Rightarrow$(iii) together with the identity \eqref{eq:faber-projection-identity}.

Assume (iii) and set $T_f=C_g\Faber_f^+$.  Then $T_f$ is onto $\Bplus$.  Proposition~\ref{prop:faber-bounded-right-inverse} gives
\[
 A_{h^{-1}}T_f=I.
\]
To prove injectivity of $A_{h^{-1}}$, let $\psi\in \Bplus$ satisfy $A_{h^{-1}}\psi=0$.  Since $T_f$ is onto, there is $\varphi\in \Bplus$ with $\psi=T_f\varphi$.  Then
\[
 0=A_{h^{-1}}\psi=A_{h^{-1}}T_f\varphi=\varphi,
\]
so $\psi=0$.  Thus $A_{h^{-1}}$ is bijective.  Its inverse is necessarily the already constructed right inverse $T_f$, which proves (ii) and \eqref{eq:faber-projection-identity}.
\end{proof}

\begin{proof}[Proof of Theorem~\ref{thm:Faber-characterization-intro}]
If $\Faber_f^+$ is bounded, Theorem~\ref{thm:bounded-faber-implies-chord-arc}
proves that $\Gamma$ is chord-arc under the sole initial assumption
that it is a rectifiable quasicircle. Thus \textup{(ii)} implies
\textup{(i)} without any prior strong quasisymmetry assumption.

If $\Gamma$ is chord-arc, the argument leading to
\eqref{eq:h-beta-alpha} gives $h\in\SQS(\T)$. Theorem~\ref{thm:main-characterization}
then gives bounded invertibility of $A_h$ and $A_{h^{-1}}$.
We may now apply Theorem~\ref{thm:faber-projection-equivalence} to obtain
that $\Faber_f^+$ is a bounded isomorphism and
\[
 C_g\Faber_f^+=(A_{h^{-1}})^{-1}.
\]
This proves \textup{(i)}$\Rightarrow$\textup{(iii)} and the additional
assertions. Finally, \textup{(iii)} implies \textup{(ii)} by definition.
All three conditions are therefore equivalent.
\end{proof}

\end{document}